\documentclass[11pt]{article}
\usepackage{graphicx,color,amsfonts,amsmath,amssymb,enumerate}
\usepackage{url,fancyhdr,indentfirst}
\usepackage{amsthm,natbib,comment,bm}
\usepackage{bbm}
\usepackage{booktabs}
\usepackage{multirow}
\usepackage{float}
\usepackage{caption}
\usepackage{subcaption}
\usepackage{threeparttable}
\usepackage{cite}
\usepackage{siunitx}
\usepackage{mathtools}
\usepackage{pgfplots}
\usepackage{nccmath}
\usepackage{wrapfig}
\usepackage{bbold}
\usepackage{mathbbol}
\usepackage{mathrsfs}
\pgfplotsset{compat=1.18}
\usepackage{dsfont}
\usepackage{tikz-qtree}
\usepackage{xr}
\usepackage[title]{appendix}

\usepackage[a4paper,margin=1in]{geometry}
\usepackage{tikz}
\usepackage[colorlinks=true,citecolor=blue]{hyperref}
\usetikzlibrary{arrows.meta,positioning,shadows.blur}

\tikzset{
  font=\small,
    box/.style={
    rectangle,draw=black,rounded corners,thick,
    minimum width=5.5cm,
    align=center,inner sep=4pt,  blur shadow={shadow blur steps=4},

      fill=gray!5 },
  arr/.style={-Latex,thick,shorten >=2pt,shorten <=2pt},
  darr/.style={arr,dashed},
  mybox/.style={
    rectangle,
    draw=black,
    line width=1.2pt,
    minimum width=2.3cm,
    minimum height=0.9cm,
    inner sep=4pt,
    align=center,
    blur shadow={shadow blur steps=4},

      fill=gray!2
  }
}
\makeatletter

\pgfarrowsdeclare{HalfBarbUp}{HalfBarbUp}
{
  \pgfarrowsleftextend{+-\pgflinewidth}
  \pgfarrowsrightextend{2.5\pgflinewidth}
}
{
  \pgfsetdash{}{0pt}
  \pgfpathmoveto{\pgfpoint{0pt}{0pt}}
  \pgfpathlineto{\pgfpoint{-4pt}{ 2.4pt}}
  \pgfpathlineto{\pgfpoint{-3pt}{ 0pt}}
  \pgfusepathqfill
}

\pgfarrowsdeclare{HalfBarbDown}{HalfBarbDown}
{
  \pgfarrowsleftextend{+-\pgflinewidth}
  \pgfarrowsrightextend{2.5\pgflinewidth}
}
{
  \pgfsetdash{}{0pt}
  \pgfpathmoveto{\pgfpoint{0pt}{0pt}}
  \pgfpathlineto{\pgfpoint{-4pt}{-2.4pt}}
  \pgfpathlineto{\pgfpoint{-3pt}{ 0pt}}
  \pgfusepathqfill
}
\makeatother

\newcommand{\E}{\mathbb{E}}
\newcommand{\R}{\mathbb{R}}

\newcommand{\M}{\mathcal{M}}

\newcommand{\N}{\mathbb{N}}
\newcommand{\p}{\mathbb{P}}

\newcommand{\esssup}{\mathrm{ess\mbox{-}sup}}
\newcommand{\essinf}{\mathrm{ess\mbox{-}inf}}
\renewcommand{\ge}{\geqslant}
\renewcommand{\le}{\leqslant}
\renewcommand{\geq}{\geqslant}
\renewcommand{\leq}{\leqslant}
\renewcommand{\epsilon}{\varepsilon}

\theoremstyle{plain}
\newtheorem{theorem}{Theorem}
\newtheorem{assumption}{Assumption}

\newtheorem{lemma}{Lemma}
\newtheorem{proposition}{Proposition}

\theoremstyle{definition}

\theoremstyle{remark}

\theoremstyle{definition}

\renewcommand{\cite}{\citet}

\newcommand{\Lip}{\mathrm{Lip}}

\usepackage[onehalfspacing]{setspace}

\renewcommand{\P}{\mathbb{P}}

\newcommand{\sgn}{\mathrm{sgn}}
\newcommand{\betaols}{\widehat{\boldsymbol{\beta}}_{{\rm ols}}}

\title{Wasserstein Distributionally Robust Quantile Regression}

\author{
  Chunxu Zhang$^{*}$ \quad Tiantian Mao$^{*}$ \quad Ruodu Wang$^{\dagger}$ \\
  $^{*}$ Department of Statistics and Finance, School of Management, \\
  University of Science and Technology of China\\
  Hefei, Anhui, China. \\
$^{\dagger}$ Department of Statistics and Actuarial Science,\\
  University of Waterloo \\
  Waterloo, Ontario, Canada.\\
 E-mails:  \href{mailto:zhangcx0627@mail.ustc.edu.cn}{zhangcx0627@mail.ustc.edu.cn}\\   \href{mailto:tmao@ustc.edu.cn}{tmao@ustc.edu.cn}, \href{mailto:wang@uwaterloo.ca}{wang@uwaterloo.ca}.
}

\date{\today}

\begin{document}

\maketitle
\begin{abstract}
We study distributionally robust quantile regression using type-$p$ Wasserstein ambiguity sets. We derive a closed-form expression for the worst-case quantile regression loss under general $p$-Wasserstein uncertainty. 
We further give a uniqueness result showing that
for $p>1$, the check loss yields the only class of convex loss functions for which such an additive Wasserstein regularization holds. 
Our analysis also uncovers  qualitative differences between the regimes $p=1$ and $p>1$. 
When $p>1$, the slope coefficients coincide with those of the regularized formulation, while the intercept undergoes a radius-dependent adjustment; 
the value $p$ affects only this intercept correction, whereas the choice of transport norm influences both. 
Finally, we establish finite-sample out-of-sample risk guarantees of order $O(N^{-1/2})$ under mild moment conditions. Numerical experiments illustrate the theoretical findings and the practical implications of the proposed formulation.
\end{abstract}

\bigskip
\noindent\textbf{Keywords:} Distributionally robust optimization, Quantile regression, Wasserstein distance, Regularization, Out-of-sample guarantees

\section{Introduction}

Quantile regression (QR), introduced by \citet{koenker1978quantile}, generalizes classical linear regression by modeling conditional quantiles of the response variable. It has become a standard tool in econometrics, finance, operations research, and machine learning due to its ability to capture distributional heterogeneity beyond the mean. QR is widely used in applications where tail behavior is of primary interest \citep{adrian2019vulnerable, chernozhukov2005iv, engle2004caviar, machado2019quantiles}. For example, in inventory management, the classical newsvendor problem reduces to estimating a demand quantile; in financial regulation, the Value-at-Risk (VaR),  a risk metric adopted in the Basel Accords,  is defined as a quantile of portfolio losses; and in macroeconomic analysis, systemic real and financial risks are quantified using tail quantiles of GDP growth and financial indicators.

Quantile regression is typically formulated in a statistical learning framework,
where the covariate–response pair \((\mathbf X, Y)\) is assumed to follow a joint
distribution \(F\) on \(\mathbb R^d \times \mathbb R\).
For $\alpha\in (0,1)$, the population objective is defined as the minimization of the expected quantile
loss under the distribution \(F\), that is,
$$
\inf_{\boldsymbol{\beta},s}\mathbb E^F[\ell_\alpha(Y - \boldsymbol{\beta}^\top\mathbf X-s)],
$$
where  \(\ell_\alpha(u) =u (\alpha - \mathbbm{1}_{\{u<0\}} )\) is known as the \emph{check loss function}. Here $s$ denotes an intercept term; for fixed $\boldsymbol{\beta}$,
the optimal $s$ corresponds to the $\alpha$-quantile
of $Y - \boldsymbol{\beta}^\top \mathbf{X}$.
In data-driven settings, the underlying distribution $F$ is typically unknown and must be inferred    from a finite dataset $\{(\mathbf X_i, Y_i)\}_{i=1}^N$ of independent and identically distributed ($i.i.d.$) samples drawn from $F$.
Classical quantile regression minimizes the empirical risk $\E^{\widehat{F}_N}[ \ell_\alpha(Y - \boldsymbol{\beta}^\top\mathbf X-s)]$, where $\widehat{F}_N = \frac1N\sum_{i=1}^N\delta_{({\bf X}_i,Y_i)}$ denotes the empirical distribution.
Decisions based solely on the empirical distribution may suffer from poor out-of-sample performance.
Regularization techniques are commonly employed to mitigate this issue by controlling model complexity and reducing overfitting \citep{koenker2004quantile,belloni2011l1,wang2012quantile}.
Distributionally robust optimization (DRO) offers an alternative and principled approach to addressing this challenge. Rather than optimizing performance under a single empirical law, DRO seeks decisions that perform well against a family of probability distributions that are statistically close to the empirical distribution, leading   to a minimax formulation

\begin{align}
\label{eq:main0}
\inf_{\boldsymbol\beta,s}\sup_{F\in\mathbb{B}(\widehat{F}_N,\epsilon)} \mathbb{E}^F[\ell_\alpha\big(Y - \boldsymbol{\beta}^\top\mathbf X-s\big)],
\end{align}
where
$\mathbb{B}(\widehat{F}_N,\epsilon)$ is an ambiguity set consisting of all distributions within an $\epsilon$-ball around $\widehat{F}_N$, with respect to a metric to be specified.
The inner maximization problem corresponds to a worst-case expectation, and any
distribution attaining the supremum is referred to as a worst-case distribution.  The choice of the underlying probability metric plays a critical role in shaping the ambiguity set. Early DRO formulations often employ moment-based ambiguity sets that characterize uncertainty through constraints on distributional moments \citep{bertsimas2010models,delage2010distributionally,xu2018distributionally}. Divergence-based ambiguity sets, such as those induced by the Kullback–Leibler divergence \citep{KL51}, impose absolute continuity with respect to the empirical distribution, thereby restricting candidate distributions to share the same support and potentially limiting out-of-sample generalization \citep{BDDM13,HH13,S17}. In contrast, Wasserstein-based ambiguity sets allow for distributions with different supports—discrete or continuous—and have been shown to exhibit favorable measure concentration and generalization properties in data-driven optimization. Motivated by these advantages, we adopt a Wasserstein DRO framework \citep{KR58,KENS19} in this work.

For two distributions $F$ and $G$ on $\R^{d+1}$ and $p\in [1,\infty]$, the type-$p$ Wasserstein distance between $F$ and $G$ is defined as
    \begin{equation*}
{\mathrm{W}_p}(F, G) :=
\begin{cases}
\inf_{\Pi} \left\{ \big(\E^{\Pi}\left[\left\|({\bf X}_1,Y_1) - ({\bf X}_2,Y_2)\right\|^p \right] \big)^{1/p} \right\}, & \text{for } p \in [1, \infty) \\
\inf_{\Pi} \left\{\esssup^\Pi \left\|({\bf X}_1,Y_1)- ({\bf X}_2,Y_2)\right\| \right\}, & \text{for } p = \infty,
\end{cases}
\end{equation*}
where $\Pi$ is a joint distribution of $({\bf X}_1,Y_1)$ and $({\bf X}_2,Y_2)$ with marginals $F$ and $G$, respectively, and $\|\cdot\|$ is a norm on $\R^{d+1}$. The type-$p $ Wasserstein ball centered at $\widehat{F}_N$  with  radius $\varepsilon$ is then defined as
   \begin{align} \label{eq:W-ball}
       \mathbb{B}_p(\widehat{F}_N,\varepsilon) =  \{F\in \mathcal M(\R^{d+1}):{\mathrm{W}_p}(F,\widehat{F}_N)\le \varepsilon\},
    \end{align}
    where $\mathcal M(\R^{d+1}) $ is the set of all distributions supported on $\R^{d+1}.$

A growing body of literature has shown that, in certain settings,    Wasserstein DRO  admits the same reformulation as the regularized regression and yields equivalent solutions \citep{SMK15,shafieezadeh2019regularization,S98,K07,CSS11}. Notable examples include logistic regression \citep{SMK15}, least squares regression \citep{blanchet2019robust,wu2022generalization}, and HO ($c$-insensitive) regression \citep{DBKSV97,LM01,LHH05},
which provide a distributionally robust interpretation and perspective of classical regularized regression methods. However, these equivalence results are almost exclusively derived for objectives of the form $\sup_{F} \mathbb{E}^F[\ell(Z)]$, and crucially rely on the loss function $\ell$ being a Lipschitz continuous function to the power $p$ when the ambiguity set is defined by a type-$p$ Wasserstein ball.\footnote{The only exception is $\ell({\bf x}) = |\boldsymbol{\beta}^\top{\bf x}| $ or $\ell({\bf x}) =  \boldsymbol{\beta}^\top{\bf x} $ for some ${\boldsymbol{\beta}}\in\R^d$, see \cite{CP18}, \cite{wu2022generalization}. }
Quantile regression falls  outside this framework. For higher-order Wasserstein ambiguity sets ($p>1$), the existing equivalence results break down due to the mismatch between the regularity of the loss and the geometry of the ambiguity set \citep{wu2022generalization}.
As a result, it remains unclear whether distributional robustness can  yield any nontrivial or interpretable effect in quantile regression.

In this paper, we show that, perhaps surprisingly, distributionally robust quantile regression (DR-QR) does admit a regularized reformulation under general type-$p$ Wasserstein ambiguity. Beyond establishing this equivalence, we further clarify its fundamental scope.
In particular, we show that such a regularization phenomenon is specialized to QR: for $p>1$, the check loss is the only convex loss function for which the distributionally robust objective admits an equivalent regularized formulation.
This impossibility result provides a sharp characterization and highlights the exceptional role of quantile regression within the Wasserstein DRO framework.

Although the resulting regularized formulation and the
original distributionally robust problem share the same in-sample loss,
they generally admit different optimal solutions when $p>1$, which in turn
leads to distinct out-of-sample performance. We derive closed-form expressions
that characterize the exact relationship between the two optimal solutions.
Exploiting this relation, we obtain an explicit expression for the difference
between the Wasserstein out-of-sample loss and the regularized out-of-sample
loss. Our analysis shows that when the quantile estimator underestimates
(resp.\ overestimates) the target for $\alpha>1/2$ (resp.\ $\alpha<1/2$), the
Wasserstein out-of-sample loss is smaller than its regularized counterpart.
Since empirical quantile estimation is known to systematically underestimate the true
value for levels close to $1$ (see, e.g., \cite{MFE15}), Wasserstein distributionally robust
quantile regression outperforms regularized quantile regression (R-QR) for small
Wasserstein radii, a phenomenon that is further corroborated by our numerical
experiments.

To understand the source of this discrepancy, it is instructive to examine the mechanism underlying the regularized reformulation.  Although the resulting formulation resembles a regularized estimator, its origin is fundamentally different.
In fact, while the two formulations yield identical slope coefficients, they generally differ in their intercept terms.
Classical regularization controls overfitting by constraining the slope coefficients, implicitly assuming that the location of the data is reliable.
In contrast, DR-QR explicitly acknowledges uncertainty in the location of the residual distribution.
Allowing the intercept to adjust under Wasserstein perturbations separates benign location shifts from harmful structural deviations.
As a result, robustness penalizes only the slope coefficients, while location uncertainty is absorbed by the intercept.
This highlights that the regularization effect in our formulation is not imposed exogenously, but arises endogenously from distributional ambiguity.

Additionally, for Wasserstein DR-QR, we
establish finite-sample generalization guarantees that are free from the
curse of dimensionality and do not depend on the Wasserstein order $p$. Moreover,
by leveraging the explicit discrepancy between the two optimal solutions, we
derive a new generalization bound for regularized quantile regression, which,
to the best of our knowledge, has not appeared in the literature.

\paragraph{Relation with other work.} A growing literature studies DRO formulations of QR. Among them, \citet{qi2021robust} investigate a Wasserstein-DRO model under \emph{fixed design} and impose the unconventional restriction that all quantile levels share a single coefficient vector~$\boldsymbol{\beta}$, i.e.,
\[
Q_\alpha(Y\mid X)=s_\alpha+X^\top\boldsymbol{\beta}\qquad\text{for all } \alpha.
\]
They obtain closed-form solutions under a type-1 Wasserstein ball.
While elegant, their formulation fundamentally differs from the standard quantile regression model, where \emph{each} quantile level $\alpha$ has its own coefficient vector $\boldsymbol{\beta}_\alpha$, and from the widely used \emph{random design} framework. Consequently, their results do not extend to general linear QR with quantile-specific coefficients under random design—arguably the most common setting in practice.
In this paper, we consider the standard linear quantile regression model
\[
Q_\alpha(Y\mid X)=s_\alpha+X^\top \boldsymbol{\beta}_\alpha,
\]
where the covariates follow a random design and each quantile level has its own coefficient vector.
The Wasserstein-DRO version of this problem—especially under general $p$-Wasserstein balls, $p\ge1$—has not been explicitly characterized. In particular, no closed-form expressions were previously available for the worst-case quantile regression loss under random design and general~$p$. This paper fills this gap.

\paragraph{Our Contributions.} This paper makes the following advances:

\begin{itemize}
\item[1.] We show that distributionally robust quantile
regression admits an exact regularized reformulation under general type-$p$
Wasserstein ambiguity.
\item [2.] For $p>1$, we show that DR-QR and its regularized reformulation share the same slope coefficients, yet differ systematically in the intercept through an explicit radius-dependent correction term (cf.\ Eq.~(\ref{eq:relationbeta0})).
This ``location adjustment'' mechanism explains why the two formulations can exhibit different out-of-sample behavior despite having identical in-sample objectives, and yields a closed-form expression for the out-of-sample risk gap (cf.\ Eq.~(\ref{eq:diffout})).

\item[3.] We establish a sharp impossibility result: for $p>1$, the
check loss is the \emph{only} convex loss function for which such an equivalence
can hold.
This result precisely delineates the scope of regularization phenomena in
Wasserstein DRO and highlights the unique role of quantile regression.

\item[4.] We establish finite-sample bounds of order $O(1/\sqrt{N})$ for both random-design and fixed-design DR-QR under only finite-moment  assumptions, thereby improving upon standard exponential-tail conditions.

\end{itemize}
Together, these results provide a unified theoretical framework and practical tools for Wasserstein DR-QR.
\paragraph{Notation.}
 Let  $(\Omega, \mathcal A,\p)$ be an atomless probability space.  A random vector ${\bf X}$ is a measurable mapping from $\Omega$ to $\R^{d}$, $d\in\N$. Denote by $F_{\bf X}$ the distribution of  $\boldsymbol X$ under $\p$. Denote by $  \mathcal{M}(\mathbb{R}^d)$ the set of all distributions on $\R^{d}$.
For $p\ge 1$, let $L^p:=L^p(\Omega, \mathcal A,\p)$ be the set of all random variables with finite $p$-th moment and
$\M_p(\R^{d})$ be the set of all distributions on $\R^{d}$ with finite $p$-th moment in each component. For any norm $\|\cdot\|$  on $\mathbb{R}^d$, its dual norm $\|\cdot\|_*$ is defined as $\|\mathbf{y}\|_*=\sup _{\|\mathbf{x}\| \leqslant 1} \mathbf{x}^{\top} \mathbf{y}$.
Let $q$ denote the H\"{o}lder conjugate of $p$, i.e., $1/p+1/q=1$. We use $\esssup Z$ and $\essinf Z$ to represent the essential supremum and the essential infimum of the random variable $Z$, respectively, i.e., $\esssup Z=\inf\{t: F_Z(t)= 1\}$ and $\essinf Z=\sup\{t:F_Z(t)= 0\}$.
For a real number $x\in\R$, we use $x_+=\max\{x,0\}$ and $x_-=\max\{-x,0\}$; and for $m\in \N$, denote by $[m]=\{1,\ldots,m\}.$ We define $\sgn(x)=\mathbbm 1_{\{x>0\}}-\mathbbm 1_{\{x<0\}}$.  For $\mathbf x\in\R^{d+1}$, let  $\delta_{\mathbf x}$ denote the degenerate distribution at $\mathbf x$.

\section{Tractable reformulation}

\label{sec:2}

We begin by developing the analytical foundation of DR-QR under type-$p$ Wasserstein ambiguity.
Specifically, we derive a  tractable reformulation of the distributionally
robust quantile regression problem \eqref{eq:main0} with  ambiguity set being the type-$p$ Wasserstein ball defined by \eqref{eq:W-ball}.

Although the problem \eqref{eq:main0} is motivated by data-driven purposes where  the ambiguity set is centered at the empirical distribution, our main result in this section holds for any centered distribution. Let $F_0$ be a distribution on $\R^{d+1}$ an$\mathbb{B}_p(F_0,\varepsilon)$ be a $p$-Wasserstein ball centered at $F_0$ defined by \eqref{eq:W-ball}. The corresponding
\emph{Wasserstein DR-QR} problem \eqref{eq:main0} is
formulated as
\begin{align}\label{OP}
    \inf_{\boldsymbol{\beta}\in\R^d,\; s\in\R}~
    \sup_{F\in \mathbb{B}_p(F_0,\varepsilon)}
        \E^F \Big[ \ell_{\alpha}\big(Y - \boldsymbol{\beta}^\top \mathbf{X}-s\big)\Big].
\end{align}
Here and throughout the paper, $\E^F$ denotes expectation under the distribution $F$ of the random vector  $(Y,{\bf X})$.
Sometimes we have $F_0\in \mathcal M(\R^k)$ for a different integer $k\ne d+1$, and  in that case $B_p(F_0,\varepsilon)$
is a subset of $\mathcal M(\R^k)$.

 \subsection{Convex reformulation}
 \label{sec:2-1}
In this section, we analyze the DR-QR problem \eqref{OP}
and derive an exact convex reformulation.   To state the result, for  $p\geq 1$ and $\alpha\in (0,1)$,
define \begin{align}\label{eq:C_alphap}
        c_{\alpha,p} = \begin{cases}
            \max\{\alpha,1-\alpha\}, & p=1,\\
              \left( \alpha^{q}(1-\alpha) + \alpha(1-\alpha)^{q} \right)^{{1}/{q}}, & p\in(1,\infty),\\
              2\alpha(1-\alpha), &p=\infty,
        \end{cases}
    \end{align}
and let $\overline{\boldsymbol{\beta}}=(\boldsymbol{\beta},-1)\in\R^{d+1}$ for $\boldsymbol{\beta}\in\R^d$ for brevity.
\begin{theorem}
    \label{thm:quantile}
    For $p\geq 1$, $\varepsilon\ge 0$, $F_0\in \mathcal M(\R^{d+1})$, the  problem  \eqref{OP} is equivalent to the following convex program
    \begin{align}\label{prob-1}
        \min_{(\boldsymbol{\beta},\overline{s})\in\R^{d+1}}  \left\{\E^{F_0} \left[ \ell_{\alpha}(Y - \boldsymbol{\beta}^\top \mathbf{X}-\overline{s}) \right] + c_{\alpha,p}\varepsilon\|\overline{\boldsymbol{\beta}}\|_* \right\},
    \end{align}
 in the sense that they share the same optimal value and that if $(\boldsymbol{\beta}^*,\overline{s}^*)$  solves \eqref{prob-1},   then $(\boldsymbol{\beta}^*,s^* )$  is the optimal solution to the problem \eqref{OP}, where
    \begin{align} \label{eq:relationbeta0}
       s^*  =\begin{cases}
            \overline{s}^*, & p=1,\\
            \overline{s}^*+\frac{\varepsilon}{q} \left(\alpha^{q}-(1-\alpha)^{q}\right) c_{\alpha,p}^{1-q}\|\overline{\boldsymbol{\beta}^*}\|_*, & p\in(1,\infty].
         \end{cases}
    \end{align}

\end{theorem}

Theorem~\ref{thm:quantile} reveals that distributional robustness naturally induces a regularization term of the form
 $c_{\alpha,p} \|\overline{\boldsymbol{\beta}}\|_*$.
The equivalence between problems \eqref{OP} and \eqref{prob-1} may be surprising, since for
$p>1$ and $\alpha\neq 1/2$ it does \emph{not} hold that
$$\sup_{F\in \mathbb{B}_p(F_0,\varepsilon)}
        \E^F \left[ \ell_{\alpha}(Y - \boldsymbol{\beta}^\top \mathbf{X}-s) \right] = \E^{F_0} \left[ \ell_{\alpha}(Y - \boldsymbol{\beta}^\top \mathbf{X}- s) \right] + c_{\alpha,p}\varepsilon\|\overline{\boldsymbol{\beta}}\|_*. $$
Instead, the regularization effect emerges only after optimizing over the intercept.
Specifically, for any  $\boldsymbol{\beta}\in\R^d$,
\begin{align}\label{eq:eqi-key}
    \inf_{s}\sup_{F\in \mathbb{B}_p(F_0,\varepsilon)}
        \E^F \left[ \ell_{\alpha}(Y - \boldsymbol{\beta}^\top \mathbf{X}-s) \right] = \inf_{s}  \E^{F_0} \left[ \ell_{\alpha}(Y - \boldsymbol{\beta}^\top \mathbf{X}- s) \right] + c_{\alpha,p}\varepsilon\|\overline{\boldsymbol{\beta}}\|_* .
\end{align}
Note that $\E^F \left[ \ell_{\alpha}(Y - \boldsymbol{\beta}^\top \mathbf{X}-s) \right]$ is linear in $F$ and convex in $s$, and the set $\mathbb{B}_p(F_0,\varepsilon)$ is a convex set.
Therefore, standard minimax arguments apply, and the order of
$\inf_s$ and $\sup_F$ of the left-hand side of  \eqref{eq:eqi-key} can be interchanged.
Thus  \eqref{eq:eqi-key} is actually the  regularization for the location-adjusted objective $\inf_s\E^F[\ell_\alpha(\cdot-s)]$:
\begin{align*}
    \sup_{F\in \mathbb{B}_p(F_0,\varepsilon)}
    \inf_{s\in\R}  \E^F \left[ \ell_{\alpha}(Y - \boldsymbol{\beta}^\top \mathbf{X}-s) \right] = \inf_{s\in\R}  \E^{F_0} \left[ \ell_{\alpha}(Y - \boldsymbol{\beta}^\top \mathbf{X}-s) \right]+ c_{\alpha,p}\varepsilon\|\overline{\boldsymbol{\beta}}\|_*.
\end{align*}
It is worth emphasizing that the above regularization identity is different from existing equivalence results in the Wasserstein DRO literature.
  Most existing Wasserstein DRO regularization results are of the form
\begin{equation*}
    \sup_{F\in\mathbb B_p(F_0,\varepsilon)} \E^F[\ell(Z)]
    \approx \E^{F_0}[\ell(Z)] + \varepsilon \cdot \mathrm{Lip}_p(\ell)\footnote{The exact formulation is $  \sup_{F\in\mathbb B_p(F_0,\varepsilon)} \E^{F}[\ell(Z)]
    = \left(\left(\E^{F_0}[\ell(Z)] \right)^{1/p}+ \varepsilon \cdot \mathrm{Lip}_p(\ell)\right)^p$.},
\end{equation*}
and rely critically on the assumption that the loss function is
Lipschitz of order $p$, with the Lipschitz exponent matching the order
of the Wasserstein ball. By contrast, the objective considered here involves an infimum over a location parameter,
\[
\rho^F(X):=\inf_{s\in\R}\E^F[\ell(X-s)],
\]
for which no comparable regularization results appear to be available in the existing literature. The key reason why a closed-form regularization still emerges in our setting lies in the location adjustment $s$.
While the check loss is globally Lipschitz, this alone is insufficient for the existing Wasserstein regularization theory when $p>1$.
Optimizing over $s$ absorbs pure location shifts in the response variable and thereby makes the exact additive regularization possible. The above observation naturally raises the question of whether similar regularization identities may hold for broader classes of loss functions beyond the check loss.
Perhaps surprisingly, the answer is essentially negative.
In fact, for $p>1$, the check loss turns out to be the \emph{only} convex loss function for which a worst-case optimized expectation of the form
\eqref{eq:cond-rho}
admits a Wasserstein regularization.
The following theorem formalizes this impossibility result by providing a complete characterization.
\begin{theorem} \label{thm:equivalence}

    For $p>1$, let $\ell:\R\to\R$ be a convex function and $\rho:\mathcal M_p(\R)\to \R$ be given by $ \rho^G(X):=\inf_{s\in\R}  \E^G \left[ \ell(X-s) \right]$.  Then, there exists $c=c_{\ell,p}>0$ such that
    \begin{equation}\label{eq:cond-rho}
        \sup_{G\in \mathbb{B}_p(G_0,\varepsilon)}  \rho^G(X)=\rho^{G_0}(X) + c\varepsilon
    \end{equation}
    holds for any $G_0\in\mathcal M_p(\R)$ and $\varepsilon\ge0$ if and only if
    $\ell(x) =  a\,\ell_\alpha(x-b) + d $
    for some $\alpha\in(0,1)$, $a>0$ and $b,\,d\in\R$,  where \(\ell_\alpha\) is the check loss function.
\end{theorem}

Using \citet[Theorem 5]{mao2022model} on the projection of Wasserstein balls, we have for any $\boldsymbol{\beta}\in\R^d$ and risk functional $\rho$, the  maximization over the $(d+1)$-dimensional Wasserstein ball reduces to a one-dimensional problem:
\begin{align*}
    \sup_{F\in \mathbb{B}_p(F_0,\varepsilon)}
        \rho^F (Y - \boldsymbol{\beta}^\top \mathbf{X}) =   \sup_{G\in \mathbb{B}_p(G_0,\varepsilon\|\overline{\boldsymbol{\beta}}\|_*)}
        \rho^G (Z) ,
\end{align*}
where $G_0$ is the distribution of $Y_0 - \boldsymbol{\beta}^\top \mathbf{X}_0$ with $(Y_0 ,  \mathbf{X}_0)\sim F_0$, and $ \mathbb{B}_p(G_0,\varepsilon\|\overline{\boldsymbol{\beta}}\|_*)$ is a one-dimensional Wasserstein ball centered at $G_0$ with radius $\varepsilon\|\overline{\boldsymbol{\beta}}\|_*$.
  Therefore, the  regularization  \eqref{eq:cond-rho}
holds for all $G_0\in\mathcal M_p(\R)$ and $\varepsilon\ge 0$
if and only if
\begin{align*}
    \sup_{F\in \mathbb{B}_p(F_0,\varepsilon)}
        \rho^F (Y - \boldsymbol{\beta}^\top \mathbf{X}) =   \rho^{F_0}(Y - \boldsymbol{\beta}^\top \mathbf{X}) + c\varepsilon  \|\overline{\boldsymbol{\beta}}\|_*
\end{align*}
holds for all $\boldsymbol{\beta}\in\R^d$, $F_0\in\mathcal M_p(\R^{d+1})$, and $\varepsilon\ge0$.
Using a standard minimax argument, this is again equivalent to the equality
\begin{align*}
    \inf_{{\boldsymbol{\beta}\in \mathcal D},\,s\in\R}~\sup_{F\in \mathbb{B}_p(F_0,\varepsilon)}
        \E^F [\ell(Y - \boldsymbol{\beta}^\top \mathbf{X}-s)] =      \inf_{{\boldsymbol{\beta}\in \mathcal D},\,s\in\R}\left\{\E^{F_0} [\ell(Y - \boldsymbol{\beta}^\top \mathbf{X}-s)] + c\varepsilon  \|\overline{\boldsymbol{\beta}}\|_*\right\}
\end{align*}
 for all $\mathcal D\subseteq\R^d$, $F_0\in\mathcal M_p(\R^{d+1})$ and $\varepsilon\ge0$. Consequently, Theorem~\ref{thm:equivalence} implies that,
for $p>1$, quantile regression is essentially the unique convex regression model
within the class of location-adjusted objectives
that admits an additive Wasserstein regularization.

\subsection{Worst-case distributions} \label{sec:worst-case}
To further understand the regularization effect induced by Wasserstein robustness, in this section, we characterize the worst-case distributions that solve the inner maximization problem of  \eqref{OP}, that is,
\begin{equation}
    \label{sup}
    \sup_{F\in \mathbb{B}_p(F_0,\varepsilon)}
    \E^F \left[ \ell_{\alpha}(Y - \boldsymbol{\beta}^\top \mathbf{X}-s) \right].
\end{equation}
To state the result, let $\mathbf{v}_{\boldsymbol{\beta}}\in\mathbb{R}^{d+1}$ satisfy
\[
\|\mathbf{v}_{\boldsymbol{\beta}}\|= 1
\quad\text{and}\quad
(-\boldsymbol{\beta},1)^\top \mathbf{v}_{\boldsymbol{\beta}}=\|(-\boldsymbol{\beta},1)\|_*
.
\]
For $p=1$, the structure of a worst-case distribution depends on whether $\alpha$ is above or below $1/2$. We first consider the case $\alpha>1/2$. 

\begin{proposition}[Worst-case distribution for $p=1$]
    \label{prop:worst-case-p1}
    For $p=1$, $\alpha\geq1/2$, $(\boldsymbol{\beta},s)\in \R^{d+1}$, and $\left(\mathbf{X}_0,Y_0\right)\sim F_0$, if $\P\left(Y_0 - \boldsymbol{\beta}^\top \mathbf{X}_0 - s \ge 0 \right)>0$, then one worst-case distribution to problem \eqref{sup}   is   the distribution of the random vector
    \begin{align*}
        \left(\mathbf{X}^*,Y^*\right) = \left(\mathbf{X}_0,Y_0\right) + \frac{\varepsilon\mathbbm{1}_{\{Y_0 - \boldsymbol{\beta}^\top \mathbf{X}_0 - s \ge 0 \}}}{\P\left(Y_0 - \boldsymbol{\beta}^\top \mathbf{X}_0 - s \ge 0 \right)}\mathbf{v}_{\boldsymbol{\beta}},
    \end{align*}
   and if $ \P\left(Y_0 - \boldsymbol{\beta}^\top \mathbf{X}_0 - s \ge 0 \right)=0$, then the supremum of  problem \eqref{sup} is not attainable.
\end{proposition}
By symmetry, for $\alpha <1/2$, if $\P\left(Y_0 - \boldsymbol{\beta}^\top \mathbf{X}_0 - s \le 0 \right)>0$, then one worst-case distribution is given by the distribution of 
\begin{align*}
    \left(\mathbf{X}_0,Y_0\right) - \frac{\varepsilon\mathbbm{1}_{\{Y_0 - \boldsymbol{\beta}^\top \mathbf{X}_0 - s \le 0 \}}}{\P\left(Y_0 - \boldsymbol{\beta}^\top \mathbf{X}_0 - s \le 0 \right)} \mathbf{v}_{\boldsymbol{\beta}} ,
\end{align*}
and if $\P\left(Y_0 - \boldsymbol{\beta}^\top \mathbf{X}_0 - s \le 0 \right)=0$, then the supremum of  problem \eqref{sup} is not attainable. 
When $\alpha=1/2$,  then the supremum of  problem \eqref{sup} is always attained by the distribution of  the following random vector
\begin{align*}
    \left(\mathbf{X}_0,Y_0\right) +  \varepsilon \left(\mathbbm{1}_{\{Y_0 - \boldsymbol{\beta}^\top \mathbf{X}_0 - s \ge 0 \}} - \mathbbm{1}_{\{Y_0 - \boldsymbol{\beta}^\top \mathbf{X}_0 - s < 0 \}} \right) \mathbf{v}_{\boldsymbol{\beta}}.
\end{align*}

\begin{proposition}[Worst-case distribution for $p=\infty$]
    \label{prop:worst-case-pinfty}
    For $p=\infty$, $\left(\mathbf{X}_0,Y_0\right)\sim F_0$ and $(\boldsymbol{\beta},s)\in \R^{d+1}$, the worst-case distribution $F^*$ to problem \eqref{sup}  is given by the distribution of the random vector
    \begin{align*}
        \left(\mathbf{X}^*,Y^*\right) = \left(\mathbf{X}_0,Y_0\right) + \varepsilon\,\operatorname{sgn}\left(Y_0 - \boldsymbol{\beta}^\top \mathbf{X}_0 - s + (2\alpha-1)\varepsilon\|\overline{\boldsymbol{\beta}}\|_*\right) \mathbf{v}_{\boldsymbol{\beta}}.
    \end{align*}
\end{proposition}

For $p\in(1,\infty)$, we provide the worst-case distribution that attains the supremum in \eqref{sup}, contingent on the optimal solution of problem \eqref{OP}.
\begin{proposition}[Worst-case distribution for $p\in(1,\infty)$]
    \label{prop:worst-case-p}
    For $p\in(1,\infty)$, $\left(\mathbf{X}_0,Y_0\right)\sim F_0,$ let $(\boldsymbol{\beta},s)=(\boldsymbol{\beta}^*,s^*)\in \R^{d+1}$ be an optimal solution to \eqref{OP}. Then, the worst-case distribution $F^*$ to problem \eqref{sup}  is given by the distribution of the random vector
    \begin{align*}
        \left(\mathbf{X}^*,Y^*\right) = \left(\mathbf{X}_0,Y_0\right) + \varepsilon c_{\alpha,p}^{1-q}\left(\alpha^{q-1}\mathbbm{1}_A
        -(1-\alpha)^{q-1}\mathbbm{1}_{A^c}\right)\frac{ \overline{\boldsymbol{\beta}^*}}{\|\overline{\boldsymbol{\beta}^*}\|},
    \end{align*}
    where $A$ is an event such that $\p(A)=1-\alpha$ and $A\subseteq \{Y_0 - \boldsymbol{\beta}^{*\top} \mathbf{X}_0 \ge {\rm VaR}_\alpha(Y_0 - \boldsymbol{\beta}^{*\top} \mathbf{X}_0)\}$ with ${\rm VaR}_\alpha^F(Z) = F^{-1}(\alpha)$.
\end{proposition}

The above characterizations reveal a unifying geometric structure across different values of $p$. For the optimal $(\boldsymbol{\beta}^*,s^*)$, any worst-case perturbation shifts the data along the direction that attains the dual norm of $\overline{\boldsymbol{\beta}^*}$—equivalently, along any unit vector $\mathbf{v}_{\boldsymbol{\beta}^*}$ satisfying 
$(-\boldsymbol{\beta}^*,1)^\top\mathbf{v}_{\boldsymbol{\beta}^*} = \|\overline{\boldsymbol{\beta}^*}\|_*$. The way probability mass is reallocated depends on $p$: when $p=1$, the adversary concentrates all mass on the worst tail events of the residual; when $p=\infty$, it shifts every point uniformly by $\pm\varepsilon\,\mathbf{v}_{\boldsymbol{\beta}^*}$ according to the residual’s sign; and for $p\in(1,\infty)$, it produces a quantile-dependent piecewise perturbation of the distribution. In each case the perturbation is chosen to maximize the impact of $\overline{\boldsymbol{\beta}^*}$ under the given norm. These constructions explain why the resulting regularization involves the dual norm $\|\overline{\boldsymbol{\beta}^*}\|_*$: the dual norm quantifies the largest possible projection of $\overline{\boldsymbol{\beta}^*}$ onto any admissible perturbation.

\section{Generalization bounds} \label{sec:3}

Suppose that $F^*$ is the data-generating distribution and let  $\widehat{F}_N= \frac{1}{N} \sum_{i=1}^N \delta_{(\widehat{y}_i,\widehat{\mathbf x}_i)}$
   be the empirical distribution based on data drawn from $F^*$, which is used as the reference distribution in Wasserstein DRO.   A key question is whether the (Wasserstein) in-sample  risk
$$\widehat{J}_N:=\inf_{s\in\R,\, \boldsymbol{\beta}\in\R^d} \sup_{F\in  {\mathbb B}_p(\widehat{F}_N,\varepsilon_N)}\E^{F}\left[\ell_\alpha(Y-{\boldsymbol{\beta}}^\top  \mathbf{X} - s)\right]=\sup_{F\in  {\mathbb B}_p(\widehat{F}_N,\varepsilon_N)}\E^{F}\left[\ell_\alpha(Y-\widehat{\boldsymbol{\beta}}_N^\top  \mathbf{X} - \widehat{s}_N)\right] ,$$
can provide an upper confidence bound on the (Wasserstein)  out-of-sample risk
$$
J_{\rm oos}:=\E^{F^*}\left[\ell_\alpha(Y-\widehat{\boldsymbol{\beta}}_N^\top  \mathbf{X} - \widehat{s}_N)\right] $$
  with a radius $\varepsilon_N$ that decays at a rate scaling gracefully with the dimensionality of the random vector $(Y,\mathbf X)$.
While generalization bounds for type-$p$ Wasserstein ambiguity sets may
suffer from the curse of dimensionality in general
(see, e.g., \cite{esfahani2018data}), recent advances show that this issue can be avoided for expected-value functionals by exploiting concentration of sample averages   (see e.g.,  \cite{G22,si2020quantifying,wu2022generalization}).
Importantly, the quantile regression framework considered in this paper falls exactly into this favorable class as the objective function is the expectation of a Lipschitz continuous loss applied to a linear predictor, enabling dimension-free $O(N^{-1/2})$ guarantees.
\begin{theorem}
\label{prop:finite-sample}
 For $\alpha\in (0,1)$,  $p\ge 1$, let $F^*\in\mathcal M(\R^{d+1})$ be the data generating distribution and   $\Gamma:=\E^{F^*}[\|(Y,\mathbf{X})\|^s]<\infty$ for some $s>2$. Then for any $\eta\in(0,1)$, we have
    \begin{align}\label{eq:generalizationbound}
        \P^N\left(\E^{F^*}\left[\ell_\alpha(Y-{\boldsymbol{\beta}}^\top  \mathbf{X} - s)\right] \leq \sup_{F\in \mathbb B_p(\widehat{F}_N,\varepsilon_N(\eta))}\E^F\left[\ell_\alpha(Y-{\boldsymbol{\beta}}^\top  \mathbf{X} - s)\right],\forall  \boldsymbol{\beta},s \right) \geq 1-\eta,
    \end{align}
    and thus,
     \begin{align}
     \label{eq:260119-1}
        \P^N\left(J_{\rm oos}\le \widehat{J}_N\right) \geq 1-\eta,
    \end{align}
    where
    the radius is chosen as $
        \varepsilon_N(\eta) =  c_\alpha  \log(2N+1)^{{1}/{s}}/{\sqrt{N}}
    $
    and
    \begin{align}\label{eq:c-alpha}
        c_\alpha =   \frac{\alpha\vee (1-\alpha)}{ \alpha \wedge(1-\alpha)}  \left(360\sqrt{d+2}+2\sqrt{2\log\left(\frac{3}{\eta}\right)}+\sqrt{\frac{3\Gamma}{\eta}}\frac{32}{s-2}\sqrt{\log{\left(\frac{24}{\eta}\right)}+2(d+2)}\right).
    \end{align}
\end{theorem}
Theorem \ref{prop:finite-sample} shows that the Wasserstein in-sample risk provides a uniform
upper confidence bound on the true risk across all decisions $(\boldsymbol{\beta},s)$, with a radius
decaying at the parametric rate $N^{-1/2}$.

If we employ the regularized quantile regression program
    \begin{align} \label{eq:260119-3}
        \inf_{(\boldsymbol{\beta},{s})\in\R^{d+1}}  \left\{\E^{\widehat{F}_N} \left[ \ell_{\alpha}(Y - \boldsymbol{\beta}^\top \mathbf{X}-{s}) \right] + \lambda\|\overline{\boldsymbol{\beta}}\|_* \right\},
    \end{align}
     where $\lambda\ge 0$ is  a regularization parameter, to estimate $(\boldsymbol{\beta},s)$, denoting by $(\widetilde{\boldsymbol{\beta}}_N,\widetilde{s}_N)=(\widetilde{\boldsymbol{\beta}}_N(\lambda),\widetilde{s}_N(\lambda))$ the minimizer to the problem \eqref{eq:260119-3}, then the regularized  in-sample risk is
\begin{align*}
\widetilde{J}_N (\lambda)
& := \E^{\widehat{F}_N}\left[\ell_\alpha(Y- \widetilde{\boldsymbol{\beta}}_N^\top  \mathbf{X} - \widetilde{s}_N)\right] +\lambda\|(\widetilde{\boldsymbol{\beta}}_N,-1)\|_*
\end{align*}
and the regularized   out-of-sample risk is
\begin{align}\label{eq:outloss-reg}
\widetilde{J}_{\rm oos}(\lambda):=\E^{F^*}\left[\ell_\alpha(Y-\widetilde{\boldsymbol{\beta}}_N^\top  \mathbf{X} - \widetilde{s}_N)\right] .
\end{align}
If we take
$\lambda=c_{\alpha,p}\varepsilon_N(\eta)$, then by Theorem \ref{thm:quantile}, we have
\begin{align*}
\widetilde{J}_N (c_{\alpha,p}\varepsilon_N(\eta))
&=  \inf_{(\boldsymbol{\beta},{s})\in\R^{d+1}}  \left\{\E^{\widehat{F}_N} \left[ \ell_{\alpha}(Y - \boldsymbol{\beta}^\top \mathbf{X}-{s}) \right] + c_{\alpha,p}\varepsilon\|\overline{\boldsymbol{\beta}}\|_* \right\}\\
&= \sup_{F\in  {\mathbb B}_p(\widehat{F}_N,\varepsilon_N)}\E^{F}\left[\ell_\alpha(Y-\widehat{\boldsymbol{\beta}}_N^\top  \mathbf{X} - \widehat{s}_N)\right]   = \widehat{J}_N,
\end{align*}
that is, the Wasserstein in-sample risk and the regularized in-sample risk coincide. In  sharp contrast, the out-of-sample risk $\widetilde{J}_{\rm oos}(c_{\alpha,p}\varepsilon_N(\eta))$  is different from the Wasserstein out-of-sample risk  $J_{\rm oos}$.

The key distinction, as revealed by Theorem~\ref{thm:quantile}, lies in the intercept adjustment induced by distributional robustness.
Although the slope coefficients coincide under the two formulations, the optimal intercepts differ when $p>1$.
To understand the practical consequences of this structural difference, we now compare the out-of-sample risks of the two estimators: $\widetilde{J}_{\rm oos} $ and $ J_{\rm oos}$.
Without loss of generality assume $\alpha>1/2$ and denote
$\Delta s=\widehat{s}_N-\widetilde{s}_N =\frac{\varepsilon}{q} \left(\alpha^{q}-(1-\alpha)^{q}\right) c_{\alpha,p}^{1-q}\|\overline{\boldsymbol{\beta}^*_N}\|_* $ by \eqref{eq:relationbeta0}.  By standard manipulation, one can verify that
\begin{align}
   \widetilde{J}_{\rm oos} - J_{\rm oos}
   & = \Delta s \left( \alpha - \P^{F^*}(Y-\widetilde{\boldsymbol{\beta}}_N^\top \mathbf{X}-\widetilde{s}_N<\Delta s) \right)+ \E^{F^*}\left[(Y-\widetilde{\boldsymbol{\beta}}_N^\top \mathbf{X}-\widetilde{s}_N)\mathbbm{1}_{\{Y-\widetilde{\boldsymbol{\beta}}_N^\top \mathbf{X}-\widetilde{s}_N\in[0,\Delta s)\}}\right]\notag\\
    & = \int_0^{\Delta s} \left[\alpha  -   \P^{F^*} (Y-\widetilde{\boldsymbol{\beta}}_N^\top \mathbf{X}-\widetilde{s}_N\le y) \right]\, {\rm d } y. \label{eq:diffout}
\end{align}
From the integral representation \eqref{eq:diffout}, we can find a  sufficient condition under which the proposed DRO approach achieves a smaller out-of-sample loss than the regularized model. Specifically, we have
$
J_{\rm oos} \le\widetilde{J}_{\rm oos}
$
is guaranteed if
$
\P^{F^*} (Y-\widetilde{\boldsymbol{\beta}}_N^\top \mathbf{X}-\widetilde{s}_N \le y ) < \alpha$  for all $y \in [0,\Delta s]$, that is, a uniform under-coverage over the interval $[0,\Delta s]$.

Existing studies have shown that the sample average approximation (SAA)-based quantile regression estimator does exhibit an intrinsic under-coverage bias, in regimes where the sample size is small relative to the dimension of the covariates (see, e.g., \cite{bai2021understanding}). Since standard regularization essentially operates through coefficient shrinkage, it does not adequately address this bias mechanism. In particular, when the Wasserstein radius—and hence the associated regularization parameter—is small, the regularized estimator remains close to the SAA solution and is therefore also prone to under-coverage. At the same time, a small Wasserstein radius implies a small value of $\Delta s$, which makes the above sufficient condition more likely to hold and, consequently, guarantees that the DRO approach achieves a smaller out-of-sample loss than regularization.
In Section~\ref{sec:numerical}, we further demonstrate through numerical experiments that the proposed DRO approach consistently outperforms regularized quantile regression across a wide range of parameter choices. This empirical advantage is plausibly driven by finite-sample uncertainty rather than asymptotic statistical effects.

To end this section, by noting the equality
$ \widehat{s}_N-\widetilde{s}_N= \frac{\varepsilon}{q} \left(\alpha^{q}-(1-\alpha)^{q}\right) c_{\alpha,p}^{1-q}\|(\widetilde{\boldsymbol{\beta}}_N,-1)\|_*$
and the lipschitz continuity of the check loss function,
 we immediately have the following finite-sample guarantee for the regularized quantile regression  based on \eqref{eq:260119-1}.
\begin{proposition}
\label{thm:finite-sample-regularized}
For $\eta\in (0,1)$ and  $\lambda =c_{\alpha,p}\varepsilon_N(\eta)$, under the conditions of Theorem \ref{prop:finite-sample}, it holds that
\begin{align}
     \label{eq:260119-2}
        \P^N\left(\widetilde{J}_{\rm oos}(\lambda)\le \widetilde{J}_N (\lambda)+ \rho(\widetilde{\boldsymbol{\beta}}_N)\lambda\right) \geq 1-\eta,
    \end{align}
    where
$\rho(\widetilde{\boldsymbol{\beta}}_N):=\max\{\alpha,1-\alpha\}\|(-\widetilde{\boldsymbol{\beta}}_N,1)\|_*$, $\widetilde{\boldsymbol{\beta}}_N$ is the optimal solution of $\boldsymbol{\beta}$ to problem   \eqref{eq:260119-3},  and $\widetilde{J}_{\rm oos}(\lambda)$ is defined by \eqref{eq:outloss-reg}.
\end{proposition}
Proposition~\ref{thm:finite-sample-regularized} provides a finite sample guarantee result for regularized quantile regression. Bounds of this structural form are standard in modern generalization theory \citep{bousquet2002stability,esfahani2018data}. In contrast to the classical regression literature, the focus here is not on controlling the estimation error of the regression coefficients, but rather on providing an explicit probabilistic upper bound on the expected loss evaluated under the unknown data-generating distribution.
Existing results on regularized quantile regression typically derive oracle inequalities for estimation and prediction errors under structural assumptions such as sparsity, restricted eigenvalue conditions, or correct model specification \citep{belloni2011l1}.
In contrast, our bounds provide high-probability guarantees on the expected out-of-sample loss itself: the true risk is bounded above by the regularized empirical risk plus a deviation term of order $O(n^{-1/2})$. These guarantees are obtained under mild moment conditions and do not rely on sparsity, model correctness, or specific structural assumptions.
 Thus, our results complement existing regularized quantile regression theory by offering finite-sample, model-agnostic risk control while maintaining comparable convergence rates.

\section{Fixed design quantile regression for linear model} \label{sec:4}

Another setting of QR is the fixed design setting under which the covariates are predetermined, so the resulting observations are independent but generally not identically distributed. As a consequence, standard Wasserstein distributionally robust formulations based on empirical distributions are not directly applicable.
To address this issue, \citet{qi2021robust} proposed a Wasserstein DR-QR model under fixed design. Their approach is built upon a linear data-generating assumption and follows a regress-then-robustify paradigm.
Specifically, \citet{qi2021robust} assume the linear model  $$y = \boldsymbol{\beta}^\top \boldsymbol{x} + e,$$ where $e$ denotes the error term. The covariates $\boldsymbol{x}_1,\ldots,\boldsymbol{x}_N$ are assumed to be predetermined, and the corresponding responses $\{y_i\}_{i\in[N]}$ are independent but not identically distributed. A key implication of the linear model assumption is that the same coefficient vector $\boldsymbol{\beta}$ governs all conditional quantiles of $y$ given $\boldsymbol{x}$, which allows the estimation of $\boldsymbol{\beta}$ to be separated from the subsequent quantile estimation.
Since the lack of identical distributions makes it infeasible to directly construct a Wasserstein ambiguity set for the joint distribution of $(\boldsymbol{x},y)$, the authors adopt a two-step procedure.
In the regression step, they estimate the common regression coefficient $\boldsymbol{\beta}$ using ordinary least squares. This step relies on the following standard regularity conditions on the design matrix.
\begin{assumption} \label{asp-1}
 \begin{itemize}
 \item [(i)] The design matrix ${\bf X}\in\R^{N\times d}$  of full rank;
 \item [(ii)] As sample size $N$ goes to infinity, the   matrix
$  \lim_{N\to\infty} \frac1N {\bf X}^\top {\bf X} $
exists and is finite nonsingular.
\end{itemize}
\end{assumption}
Under this assumption, the ordinary least squares estimator $\widehat{\boldsymbol{\beta}}_{\rm ols} = ({\bf X}^\top{\bf X})^{-1} {\bf X}^\top {\bf y}$  is well defined and consistent, where  ${\bf y}$ is a column vector that stores $\{y_i\}_{i\in [N]}$. In the robustify step, \citet{qi2021robust} construct adjusted residuals $z_i:=y_i- {\betaols}^{\,\top}\boldsymbol{x}_i$, $i\in [N]$, and use the empirical distribution $\widehat{F}_N^{\rm ols} := \frac1N\sum_{i=1}^N \delta_{z_i}$ as a surrogate for the distribution of the target residual  $\overline{e} :=e+(\boldsymbol{\beta}-\betaols)^\top\boldsymbol{c}$ where $\boldsymbol{c}\in\mathbb{R}^d$ is a fixed feature vector of interest. This construction yields an approximately $i.i.d.$  sample, thereby enabling the application of Wasserstein distributionally robust optimization to the residual distribution. They then consider the following Wasserstein DRO problem based on a type-1 Wasserstein ball:

\begin{align}
    \label{eq:fixed-design}
    \inf_{s\in\R} \sup_{F \in \mathbb B_p\left(\widehat{F}_N^{\rm ols},\,\varepsilon\right)} \E^{F} \left[\ell_\alpha(\overline{e}-s)\right]
\end{align}
with $p=1$. Theorem 1 of \citet{qi2021robust} shows that problem \eqref{eq:fixed-design} admits the tractable reformulation
\begin{align*}
    \inf_{s\in\R,\,\lambda\geq\alpha\vee(1-\alpha)} \lambda \varepsilon + \E^{\widehat{F}_N^{\rm ols}}\left[\ell_\alpha(\overline{e}-s)\right].
\end{align*}
However, the robustness parameter $\lambda$ and the decision variable $s$ are completely decoupled in the above formulation, which leads to the trivial optimizer $\lambda^*=\alpha\vee(1-\alpha)$. Consequently, the optimal solution $s^*$ coincides with the empirical $\alpha$-quantile of $\widehat{F}_N^{\rm ols}$ and is independent of the Wasserstein radius $\varepsilon$. As a result, despite its distributionally robust formulation, the resulting estimator does not exhibit genuine robustness with respect to distributional perturbations.

We give a regularization formulation of problem \eqref{eq:fixed-design} uniformly for general type-$p$ Wasserstein distance with $p\ge 1$ in the following proposition which follows from the same reasoning as in Theorem \ref{thm:quantile}.
\begin{proposition}
    For $p\ge 1$ and $\varepsilon\geq 0$, the problem \eqref{eq:fixed-design} admits the reformulation:
    \begin{align} \label{eq:fixed-ref}
        \inf_{\overline{s}\in\R}  \E^{\widehat{F}_N^{\rm ols}}\left[\ell_\alpha\left(\overline{e}-\overline{s}\right)\right] + c_{\alpha,p}\varepsilon,
    \end{align}
 where $c_{\alpha,p}$ is given by \eqref{eq:C_alphap}, and  $\overline{s}^*$ is the optimal solution to problem \eqref{eq:fixed-ref}
 if and only if

 \begin{align} \label{eq:beta_0-cr}
       s^* : =\begin{cases}
           \overline{s}^*, & p=1,\\
            \overline{s}^*+\frac{\varepsilon}{q} \left(\alpha^{q}-(1-\alpha)^{q}\right) c_{\alpha,p}^{1-q},& p>1,
        \end{cases}
    \end{align}
is the optimal solution to problem \eqref{eq:fixed-design}.
\end{proposition}

Let $F^*\in\mathcal M(\R)$ be the data generating distribution of the noise $e$.
In the fixed design framework, the (Wasserstein) in-sample risk is
\begin{align}
    \widehat{J}_N  :
    &= \sup_{F \in \mathbb B_p\left(\widehat{F}_N^{\rm ols},\,\varepsilon_N\right)} \E^{F} \left[\ell_\alpha(\overline{e}-\widehat{s}_N)\right] ,\notag
\end{align}
with $\widehat{s}_N$ being the optimal value of $s$ to the problem \eqref{eq:fixed-design}, and the out-of-sample risk is
$$
J_{\rm oos}:=\E^{F^*_{\overline{e}}}\left[\ell_\alpha(\overline{e}-\widehat{s}_N)\right] ,$$
where $F^*_{\overline{e}}$  is the distribution of the adjusted residuals $\overline{e}$.
\begin{theorem}
\label{prop:finite-sample-fixed}
 For $p\ge 1$, suppose Assumptions (i) and (ii) hold and   $\Gamma_0:=\E^{F^*}[|e|^s]<\infty$ for some $s>2$. Then  for any $\eta\in(0,1)$, we have
    \begin{align*}
        \P\left( J_{\rm oos} \leq \widehat{J}_N \right) \geq 1-\eta,
    \end{align*}
   Here, the radius is chosen as
    $$
    \varepsilon_N = \varepsilon_{1,N}+\frac{\alpha\vee(1-\alpha)}{\alpha\wedge(1-\alpha)}(\varepsilon_{2,N}+\varepsilon_{3,N}),
    $$
    where $\varepsilon_{1,N}:= \varepsilon_N(\eta/3) =  c_\alpha  \log(2N+1)^{{1}/{s}}/{\sqrt{N}}$ with $c_\alpha $ defined by \eqref{eq:c-alpha}, $\varepsilon_{2,N}=3 \Gamma_0^{1/s}\sqrt{d/N}/\eta $ and $\varepsilon_{3,N}=\Gamma_0^{1/s}\sqrt{3 \boldsymbol{c}^\top(\boldsymbol{X}^\top\boldsymbol{X})^{-1}\boldsymbol{c}/\eta} = O(1/\sqrt{N})$ as $N\to\infty$.
\end{theorem}

\cite{qi2021robust}  give a similar finite sample guarantee, which  achieves the $O(N^{-1/2})$ radius order matching that of Theorem \ref{prop:finite-sample-fixed},  under Assumptions (i) and (ii) and   the prohibitive condition of exponential moments: $\E^{F^*}[\exp(|e|^a)]<\infty$ for some $a>1$ ($p=1$), which is substantially stronger than our $(2+\delta)$-th moment assumption. Furthermore, their approach suffers from ``curse of the order $p$", that is, the decay rate of the radius degrades severely as $p$ increases (see e.g. Theorem 1 of \citet{wu2022generalization}), while our rate remains $O(N^{-1/2})$ uniformly for all $p\geq 1$.

\section{Numerical experiments}\label{sec:numerical}
This section provides a comprehensive empirical evaluation of DR-QR. The experiments are designed to illustrate three key theoretical insights developed earlier:
(i) the implicit regularization induced by Wasserstein robustness, including the distinct roles of the Wasserstein order $p$ and the norm in shaping the ambiguity geometry,
(ii) the intercept-adjustment mechanism that differentiates DR-QR from classical regularized QR, and
(iii) the generalization bounds.

\subsection{Loss based performance evaluation} \label{experiment:1}
In this experiment, we evaluate the performance of the DR-QR model across different training sample sizes by computing the average quantile loss on a fixed test set.

We use simulated data generated from a linear regression model $Y_i = \boldsymbol{\beta}^\top \mathbf{X}_i + e_i$ for $i=1,\ldots,N$, where $\mathbf{X}_i \in \mathbb{R}^d$ are $i.i.d.$~features from a multivariate normal distribution, the noise term $e_i\sim \mathcal{N}(0,\sigma^2)$ is independent of $\mathbf{X}_i$, and the regression coefficient vector $\boldsymbol{\beta} \in \mathbb{R}^d$ follows a random sparse structure: each component is independently activated with probability $0.2$ and assigned a random sign. To avoid degenerate cases, at least one component of $\boldsymbol{\beta}$ is ensured to be nonzero.

Throughout the experiments, we fix the quantile level to $\alpha=0.7$, the dimension to $d=30$ and the noise level to $\sigma=5$. The
test size is fixed at $10^4$ in order to approximate the expectation under the true data-generating distribution. The training sample size $N$ ranges from $20$ to $2000$ on a logarithmic scale. This range allows us to examine the robustness of the model in small-sample regimes as well as its asymptotic convergence behavior as the sample size increases.

We evaluate the DR-QR model under various configurations, varying the choice of the Wasserstein order $p$ and the norm used in the metric.
We use the standard QR as a benchmark for comparison.
For each training size $N$, the Wasserstein radius $\varepsilon$ is scaled proportionally to $N^{-1/2}$, in accordance with our theoretical results in Theorem \ref{prop:finite-sample}.

\begin{figure}[htbp]

    \caption{Out-of-sample quantile loss under different Wasserstein balls.}
    \label{fig:geometry}
     \begin{center}
    \begin{subfigure}[t]{0.48\textwidth}
   \includegraphics[width=\textwidth]{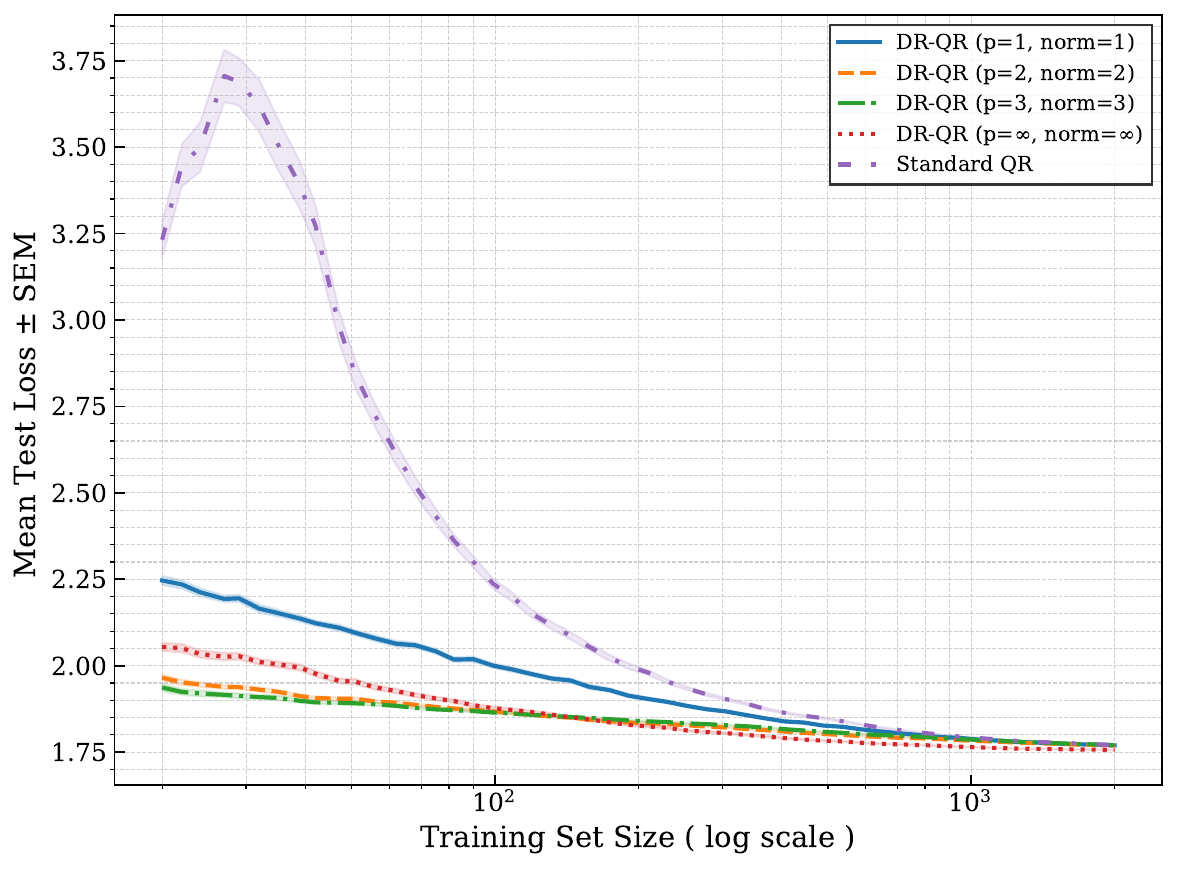}
        \caption{Different $p$ and norms.}
        \label{fig:geometry-a}
    \end{subfigure}
    \hfill
    \begin{subfigure}[t]{0.48\textwidth}

     \includegraphics[width=\textwidth]{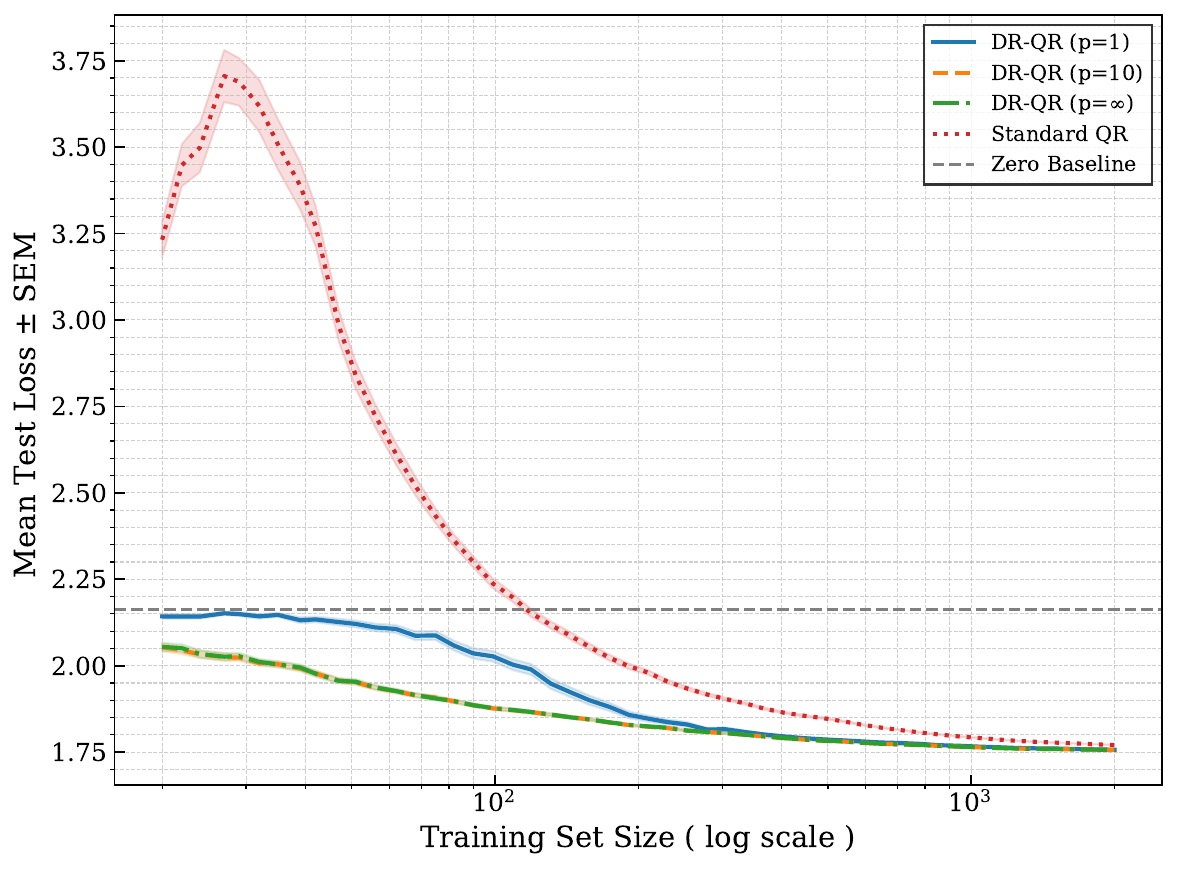}
        \caption{Varying $p$ under  $\ell_\infty$ norm.}
        \label{fig:geometry-b}
    \end{subfigure}
    \vspace{0.5em}
    \end{center}
    {\it Notes.} Shaded regions represent $\pm$ one standard error of the mean (SEM) across 100 independent runs. The Zero Baseline corresponds to the trivial estimator $\boldsymbol{\beta}=\mathbf{0}$. In Panel (a), $\texttt{norm}=k$ denotes the $\ell_k$-norm.
\end{figure}
Figure~\ref{fig:geometry}(a) compares the out-of-sample loss of DR-QR
under different Wasserstein orders $p$ and norms.
All robust specifications improve over standard QR,
particularly in small-sample regimes, while exhibiting stable convergence as $N$ increases.
Figure~\ref{fig:geometry}(b) focuses on the role of the Wasserstein order $p$
under a fixed $\ell_\infty$ norm.
The choice $p=1$ yields the poorest performance. This is consistent with the geometry of the ambiguity set:  the $\ell_\infty$-based Wasserstein ball is the largest among all $\ell_q$-induced Wasserstein balls. When combined with $p=1$, this yields the strongest effective adversary and hence the most severe shrinkage effect.
As a result, the solution is driven toward the trivial estimator $\boldsymbol{\beta}=\mathbf{0}$, which explains why the corresponding loss closely tracks the zero baseline when $N$ is small.
In contrast, $p=\infty$ produces a smaller effective ambiguity set,
mitigating this degeneracy and yielding improved performance.

Under the $\ell_\infty$ norm, the induced transport geometry is the most conservative among the norm choices considered here.
When combined with $p=1$, this yields the strongest effective adversary and hence the most severe shrinkage effect.
As a result, the solution is driven toward the trivial estimator $\boldsymbol{\beta}=\mathbf{0}$, which explains why the corresponding loss closely tracks the zero baseline when $N$ is small.

\subsection{Comparison with regularized quantile regression}
In this section, we compare our Wasserstein DR-QR model with the regularized quantile regression model,\footnote{To align with the DRO formulation, we consider the regularized quantile regression model with penalty $\|(\boldsymbol{\beta},-1)\|_{*}$. The standard regularized quantile regression instead employs $\|\boldsymbol{\beta}\|_{*}$ as the regularizer.
Under $\ell_p$ norms, these two penalties are comparable up to constants, and therefore induce a similar scale of regularization on the slope coefficient $\boldsymbol{\beta}$.}

\[
\inf_{\boldsymbol{\beta}\in\R^d, s\in\R}
\;
\frac{1}{N}\sum_{i=1}^N \ell_{\alpha}\!\left(Y_i - \boldsymbol{\beta}^\top \mathbf{X}_i - s\right)
+
\lambda \|\overline{\boldsymbol{\beta}}\|_{*}.
\]
The R-QR formulation is well established in the literature and is commonly regarded as an effective regularization approach. In contrast to DR-QR, the R-QR model does not incorporate an explicit adjustment of the intercept induced by distributional robustness. To facilitate a meaningful comparison, we align the penalty coefficient
$\lambda$ and the penalty norm with the modeling choices in the DRO formulation, so that both formulations are evaluated on a common scale.

The experimental setup largely follows that of Section~\ref{experiment:1},
with the exception that the true coefficient vector $\boldsymbol{\beta}$
is generated from $U[1,5]^d$, resulting in a larger magnitude of
$\|\overline{\boldsymbol{\beta}}\|_{*}$. We set $\alpha=0.9$, $p=2$, and the $\ell_2$-norm.
Figure~\ref{fig:4} presents the out-of-sample loss.
\begin{figure}[htbp]
    \centering
   \caption{Test quantile loss versus Wasserstein radius for DR-QR and R-QR across different training sample sizes.  }
    \begin{subfigure}[b]{0.32\textwidth}
        \centering
        \includegraphics[width=\textwidth]{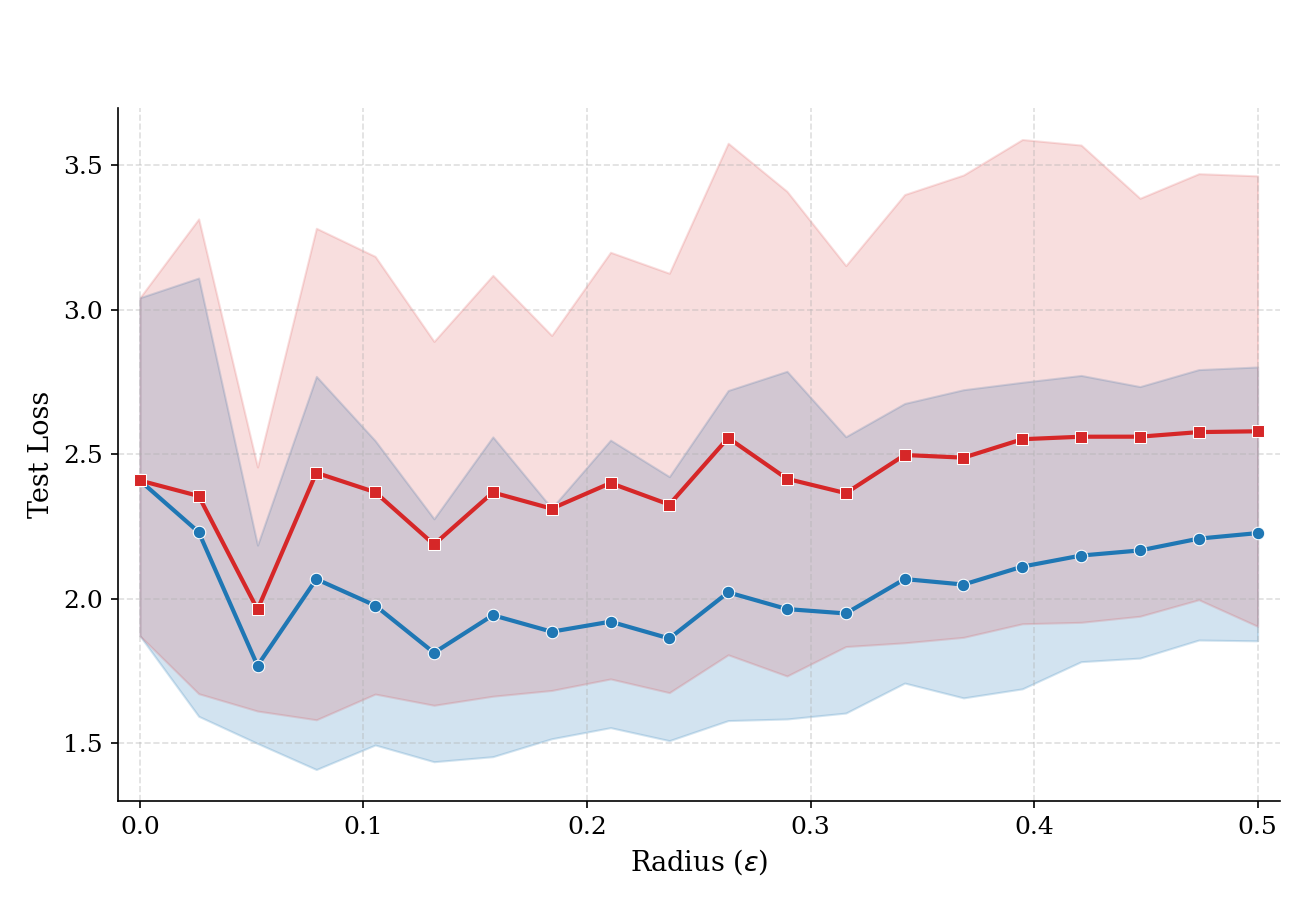}
        \caption{$N=50$}
    \end{subfigure}
    \hfill
    \begin{subfigure}[b]{0.32\textwidth}
        \centering
        \includegraphics[width=\textwidth]{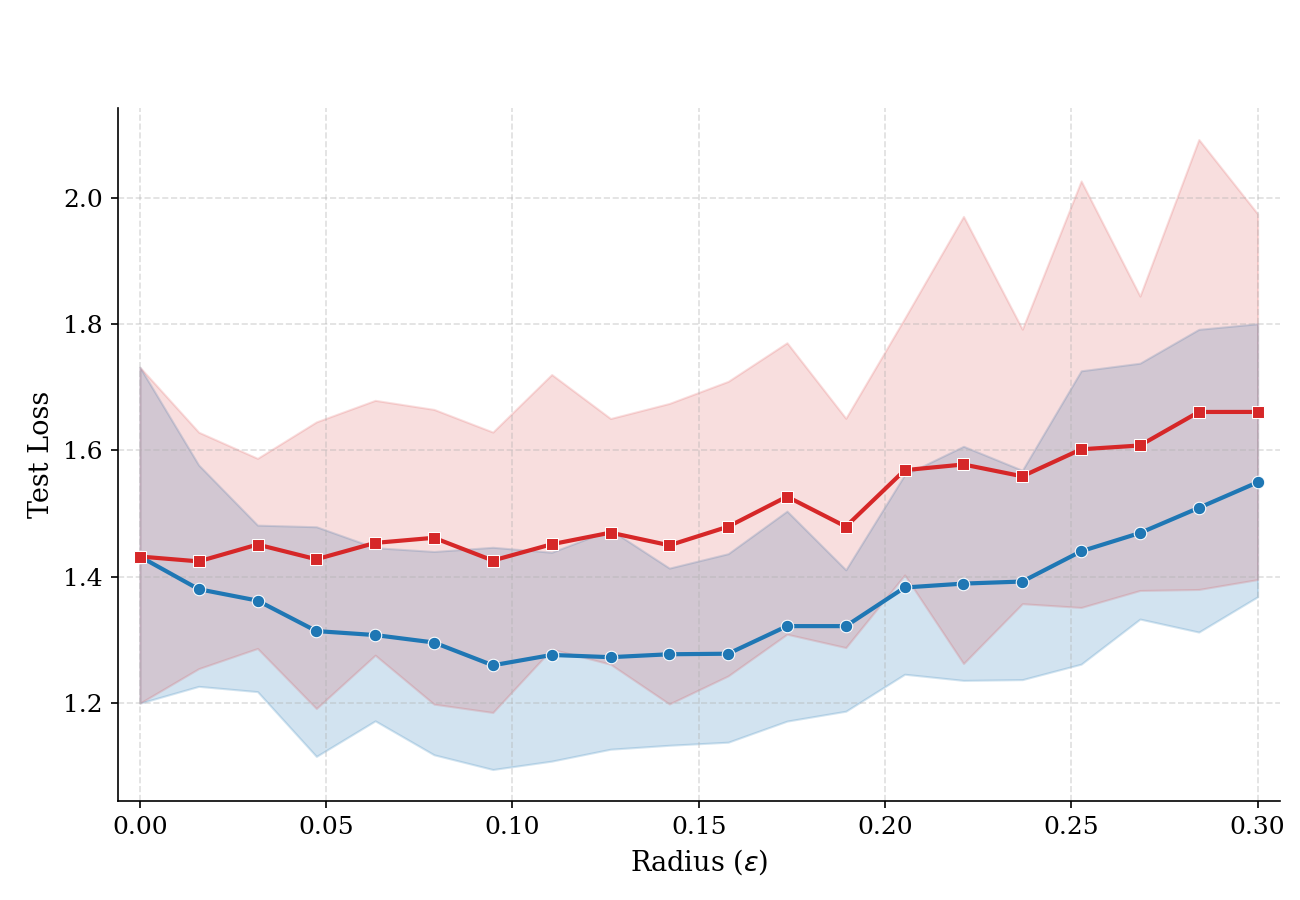}
        \caption{$N=100$}
    \end{subfigure}
    \hfill
    \begin{subfigure}[b]{0.32\textwidth}
        \centering
        \includegraphics[width=\textwidth]{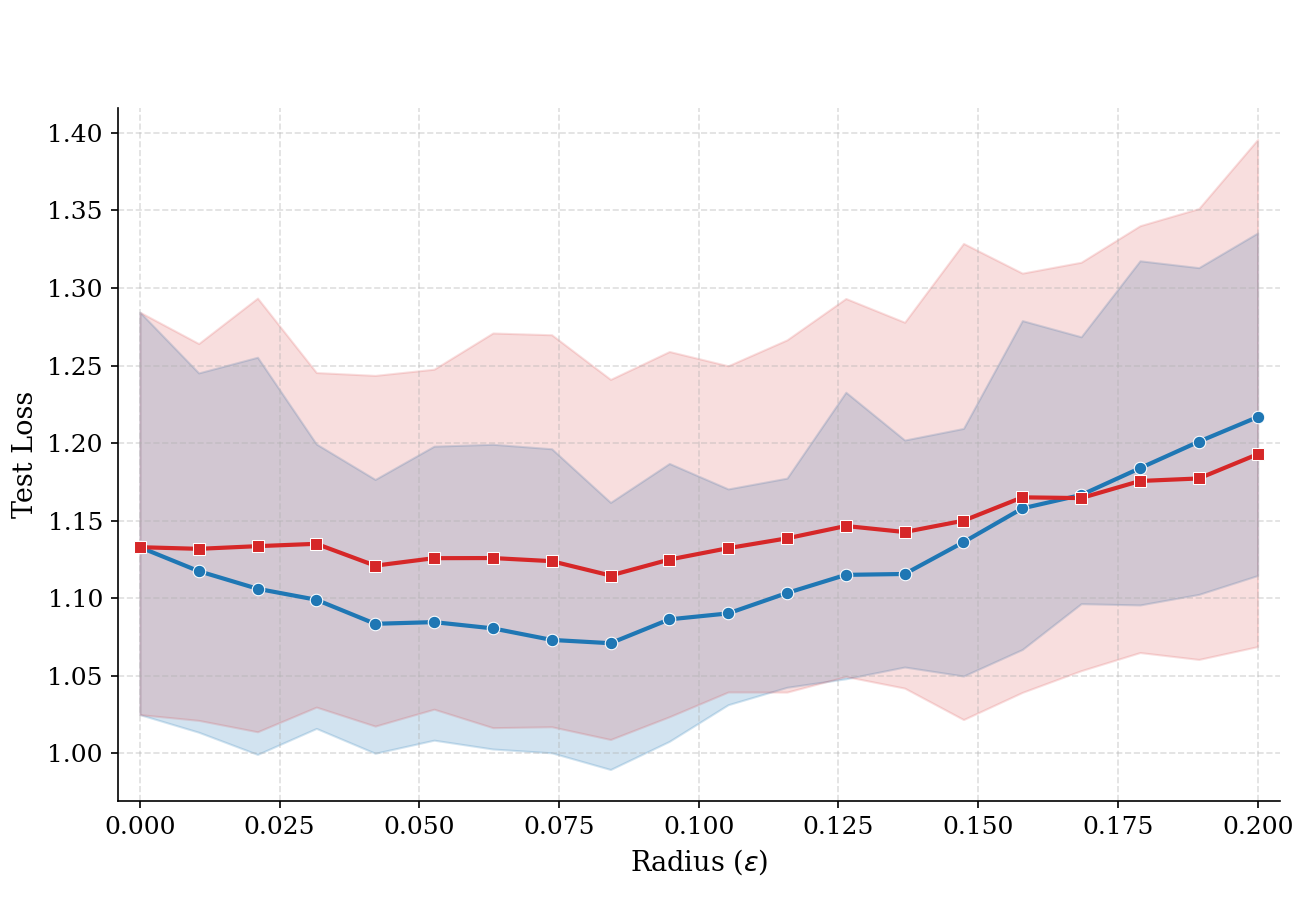}
        \caption{$N=200$}
    \end{subfigure}

    \vspace{0pt}

    \begin{center}

\includegraphics[width=0.9\textwidth]{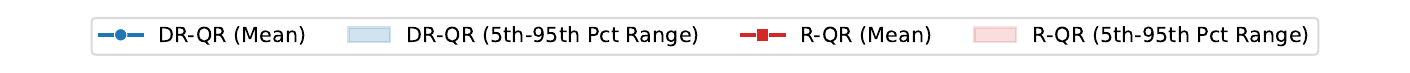}
   \label{fig:4} \end{center}
{{\it Notes.} The shaded regions represent the $5$th--$95$th percentile range across $100$ independent runs.}

\end{figure}

As shown in Figure~\ref{fig:4}, when the sample size and dimension are comparable ($N=50,100$), the DR-QR model exhibits superior performance, both in terms of the optimal test loss attained over the radius grid and in overall behavior across different radii.  When $N=200$, the gap narrows, yet DR-QR still attains lower average loss and uniformly smaller upper-tail risk.
These findings   align with the mechanism of distributional robustness. Under asymmetric loss functions, when there is limited confidence in the underlying data-generating distribution, the response estimate should be adjusted toward higher values under asymmetric loss functions that impose a larger penalty on positive residuals. While such adjustments are typically heuristic and require separate tuning of penalty strength and intercept correction, our Wasserstein-DRO formulation provides a unified, single-parameter mechanism. By selecting an optimal radius $\varepsilon$, the model simultaneously calibrates regularization strength and intercept adjustment in a principled manner.

To further illustrate this, we compare  the   estimates of the intercept term $s$ to illustrate the mechanism through which DR-QR achieves robustness.

Specifically, we investigate multiple quantile levels $\alpha \in \{0.2, 0.4, 0.7, 0.9\}$ with training sample sizes $N$ ranging from $10$ to $200$ and the Wasserstein order $p \in \{1.5, 2, 5, \infty\}$. To align with theoretical statistical convergence rates, the Wasserstein radius is scaled as $\varepsilon = N^{-1/2}$.
Table \ref{tab:intercept_comparison} compares the theoretical noise quantile $Q_{\alpha}(e) = \sigma\Phi^{-1}(\alpha)$ with the estimates obtained from the standard SAA, the regularized quantile regression (R-QR), and our proposed DR-QR model.

\begin{table}[ht]
    \centering
    \begin{threeparttable}
        \caption{Comparison of intercept estimates across different models and $p$ values.}
        \label{tab:intercept_comparison}
        \setlength{\tabcolsep}{4pt}
        \footnotesize
        \begin{tabular}{@{}c c c cc cc cc cc c@{}}
            \toprule
            & & & \multicolumn{2}{c}{$p=1.5$} & \multicolumn{2}{c}{$p=2.0$} & \multicolumn{2}{c}{$p=5.0$} & \multicolumn{2}{c}{$p=\infty$} & \\
            \cmidrule(lr){4-5} \cmidrule(lr){6-7} \cmidrule(lr){8-9} \cmidrule(lr){10-11}
            $\alpha$ & $N$ & Theoretical & DR-QR & R-QR & DR-QR & R-QR & DR-QR & R-QR & DR-QR & R-QR & SAA \\
            \midrule
            \multirow{6}{*}{0.2} & 10 & \multirow{6}{*}{-4.2081} & \textbf{-3.2083} & -1.1518 & \textbf{-2.8537} & -0.6627 & \textbf{-2.5860} & -0.6641 & \textbf{-2.4117} & -0.6376 & - \\
                                & 20 &                           & \textbf{-3.9560} & -1.9615 & \textbf{-3.7011} & -1.5974 & \textbf{-3.2241} & -1.3421 & \textbf{-3.0406} & -1.2994 & - \\
                                & 30 &                           & \textbf{-4.3863} & -2.5129 & \textbf{-4.2028} & -2.2376 & \textbf{-3.7963} & -2.0453 & \textbf{-3.5931} & -1.9663 & -0.9684 \\
                                & 50 &                           & \textbf{-4.8831} & -3.2289 & \textbf{-4.8649} & -3.1424 & \textbf{-4.5812} & -3.0508 & \textbf{-4.4221} & -3.0056 & -2.7752 \\
                                & 100 &                          & -5.1884 & \textbf{-3.9428} & -5.1942 & \textbf{-3.9140} & -4.9983 & \textbf{-3.8721} & -4.9176 & \textbf{-3.8805} & -3.8034 \\
                                & 200 &                          & -4.8952 & \textbf{-4.0065} & -4.9002 & \textbf{-3.9896} & -4.7835 & \textbf{-3.9867} & -4.7151 & \textbf{-3.9817} & -3.9640 \\
            \midrule
            \multirow{6}{*}{0.4} & 10 & \multirow{6}{*}{-1.2667} & \textbf{-1.1211} & -0.4790 & \textbf{-1.1061} & -0.4560 & \textbf{-1.0813} & -0.4415 & \textbf{-1.0972} & -0.4622 & - \\
                                & 20 &                           & \textbf{-1.0441} & -0.4007 & \textbf{-0.9869} & -0.3373 & \textbf{-0.9736} & -0.3320 & \textbf{-0.9714} & -0.3352 & - \\
                                & 30 &                           & \textbf{-1.3025} & -0.7108 & \textbf{-1.3204} & -0.7231 & \textbf{-1.2959} & -0.7050 & \textbf{-1.3115} & -0.7253 & -0.2466 \\
                                & 50 &                           & -1.5066 & \textbf{-1.0281} & -1.5055 & \textbf{-1.0236} & -1.4983 & \textbf{-1.0219} & -1.4891 & \textbf{-1.0156} & -1.0646 \\
                                & 100 &                          & -1.5114 & \textbf{-1.1616} & -1.5133 & \textbf{-1.1611} & -1.5095 & \textbf{-1.1614} & -1.5047 & \textbf{-1.1591} & -1.1634 \\
                                & 200 &                          & -1.4959 & \textbf{-1.2492} & -1.5005 & \textbf{-1.2522} & -1.4952 & \textbf{-1.2498} & -1.4959 & \textbf{-1.2525} & -1.2311 \\
            \midrule
            \multirow{6}{*}{0.7} & 10 & \multirow{6}{*}{2.6220}  & \textbf{1.7742} & 0.4982 & \textbf{1.7442} & 0.4378 & \textbf{1.5770} & 0.3297 & \textbf{1.5014} & 0.2968 & - \\
                                & 20 &                           & \textbf{2.4369} & 1.1569 & \textbf{2.3664} & 1.0477 & \textbf{2.2482} & 0.9917 & \textbf{2.1940} & 0.9768 & - \\
                                & 30 &                           & \textbf{2.8659} & 1.6621 & \textbf{2.8126} & 1.5684 & \textbf{2.6928} & 1.4983 & \textbf{2.6435} & 1.4875 & 0.9821 \\
                                & 50 &                           & 3.1859 & \textbf{2.1952} & 3.1831 & \textbf{2.1678} & 3.1374 & \textbf{2.1698} & \textbf{3.0815} & 2.1431 & 1.9727 \\
                                & 100 &                          & 3.1946 & \textbf{2.4575} & 3.2091 & \textbf{2.4559} & 3.1623 & \textbf{2.4467} & 3.1455 & \textbf{2.4517} & 2.4242 \\
                                & 200 &                          & 3.0435 & \textbf{2.5227} & 3.0557 & \textbf{2.5245} & 3.0254 & \textbf{2.5212} & 3.0109 & \textbf{2.5227} & 2.5033 \\
            \midrule
            \multirow{6}{*}{0.9} & 10 & \multirow{6}{*}{6.4078}  & \textbf{4.8489} & 1.8347 & \textbf{4.1284} & 0.9954 & \textbf{3.4085} & 1.0228 & \textbf{2.7131} & 0.6995 & - \\
                                & 20 &                           & \textbf{6.5701} & 3.5274 & \textbf{5.9189} & 2.7425 & \textbf{4.5551} & 2.0534 & \textbf{4.0859} & 2.0038 & - \\
                                & 30 &                           & \textbf{7.3579} & 4.4707 & \textbf{6.6430} & 3.6169 & \textbf{5.3259} & 2.9363 & \textbf{4.7642} & 2.7555 & 1.7933 \\
                                & 50 &                           & 7.8204 & \textbf{5.1333} & \textbf{7.2081} & 4.4360 & \textbf{6.0842} & 3.9605 & \textbf{5.5790} & 3.8018 & 3.2031 \\
                                & 100 &                          & 7.8475 & \textbf{5.6065} & 7.6314 & \textbf{5.3989} & \textbf{6.9312} & 5.2591 & \textbf{6.6060} & 5.2252 & 5.1130 \\
                                & 200 &                          & 7.8167 & \textbf{6.1880} & 7.7276 & \textbf{6.1249} & 7.2611 & \textbf{6.0718} & 7.0177 & \textbf{6.0374} & 6.0040 \\
            \bottomrule
        \end{tabular}
        \begin{tablenotes}
            \footnotesize
            \item \textit{Note:} SAA estimates are denoted by `-' for $N < d$  due to non-identifiability. Estimates with smaller absolute errors relative to the theoretical values are highlighted in bold.
        \end{tablenotes}
    \end{threeparttable}
\end{table}
The numerical results reveal distinct behaviors in intercept estimation across the three models. Specifically, SAA tends to significantly underestimate the intercept at upper quantile levels and overestimate it at lower quantile levels when the sample size is limited. While the R-QR model mitigates these biases through regularization, the sign of the bias remains unchanged relative to the theoretical quantile. In contrast, the DR-QR model exhibits an opposing trend: it proactively shifts the intercept upward at upper quantile levels and downward at lower quantile levels. Although this adjustment may lead to larger absolute estimation errors, it consistently results in lower test loss due to the asymmetric nature of the quantile loss function. Moreover, as $p$ increases, the penalty coefficient in R-QR decreases, leading to less accurate intercept estimates, whereas the intercept correction in DR-QR effectively compensates for this insufficient penalization.

\subsection{Numerical demonstration of generalization bounds}
We next provide a numerical illustration of the finite-sample generalization bounds derived in Theorem~\ref{prop:finite-sample}. In this experiment, we follow the settings of  \cite{shafieezadeh2019regularization}. In contrast to the loss-based evaluation, our primary focus here is on the behavior of the Wasserstein radius $\varepsilon$. For each candidate radius in the search space $\mathcal{E}$, we record three quantities: the radius selected by 5-fold cross-validation, the optimal radius that minimizes the test loss, and the smallest radius for which the empirical generalization inequality $J_{\rm oos}\le \widehat{J}_N$ holds with at least $95\%$ frequency across trials. The first two quantities are averaged over all trials, while the third quantity is determined by identifying the minimal radius in $\mathcal{E}$ that satisfies the empirical validity condition.

Throughout the experiment, the search space of the Wasserstein radius is chosen as
$
\mathcal{E}
=
\left\{ a \cdot 10^{-b} : a \in \{1,\ldots,10\},\; b \in \{1,2,3\} \right\}
\;\cup\;
\{1.1, 1.2, \ldots, 2.0\}.
$
We consider different training sample sizes
$
N \in \{10,\ldots,90\} \cup \{100,\ldots,900\} \cup \{1{,}000,\ldots,10{,}000\},
$
and fix the test sample size to $10^4$ in order to approximate expectations under the true data-generating distribution. Other settings largely follow that of the previous subsection, including the data-generating process, the quantile level $\alpha$, and the ambient dimension $d$. In particular, we fix the noise term to $e_i \sim \mathcal{N}(0,1)$ and the regression coefficient vector to $\boldsymbol{\beta}=\mathbf{1}/\sqrt{d}$. We consider Wasserstein distance under the $\ell_2$ norm with $p=2$. For each training size $N$, we repeat the experiment over $100$ independent trials.

By varying the training sample size $N$, this experiment allows us to examine how these radii evolve with increasing data availability. In particular, we can check whether the empirically valid bound radius exhibits the same decay behavior as the cross-validated and oracle radii, as suggested by the theoretical scaling laws. The results are summarized in Figure~\ref{fig:3}, which shows that all three radii decrease at a rate approximately proportional to $N^{-1/2}$, thereby corroborating our theoretical findings.
\begin{figure}[htbp]
    \centering
    \includegraphics[width=0.6\textwidth]{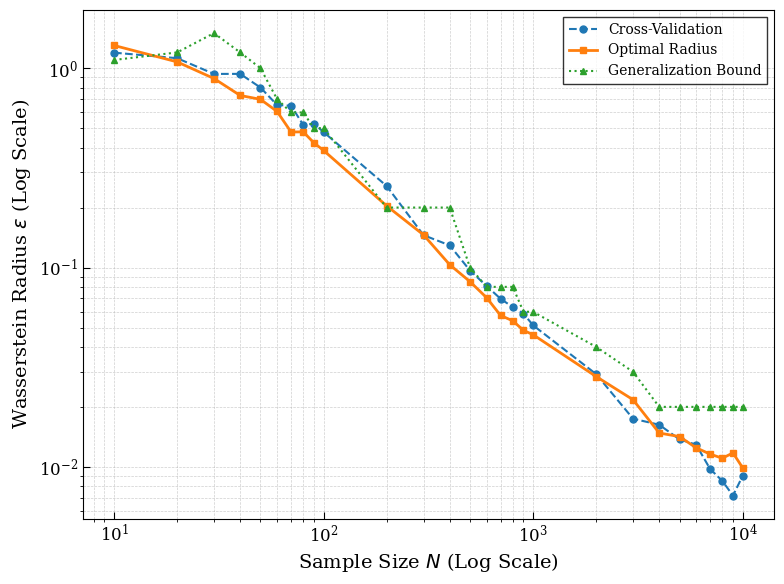}
    \caption{Wasserstein radius versus training sample size for cross-validated, optimal, and empirically valid bounds.}
    \label{fig:3}
\end{figure}

\section{Conclusion}

This paper develops a unified theoretical and computational framework for DR-QR under general type-$p$ Wasserstein uncertainty. We    derive a closed-form expression for the worst-case quantile regression loss for any $p\ge 1$, leading to transparent convex regularized reformulations that strictly extend existing DRO approaches—particularly those relying on fixed design or a shared coefficient vector across quantile levels. Our analysis uncovers sharp qualitative differences between $p=1$ and $p>1$, clarifying how the order $p$ affects only the intercept while the underlying norm in the Wasserstein metric influences both intercept and slope.  Beyond establishing this equivalence, we further show an impossibility result that the check loss is  the only convex loss function for which the distributionally robust objective admits an equivalent regularized formulation.
Finally, we establish finite-sample guarantees of order $O(N^{-1/2})$ under mild moment conditions, overcoming the curse of dimensionality. Overall, our results offer a complete and versatile foundation for robust quantile regression, with potential extensions to nonlinear models, high-dimensional settings, and multivariate quantile structures.

\setcounter{equation}{0}
\setcounter{lemma}{0}
\renewcommand{\theequation}{A\arabic{equation}}
\renewcommand{\thelemma}{A\arabic{lemma}}

\begin{appendices}

\section{Proofs of the Main Results}
\subsection{Proofs of Section \ref{sec:2}}
\subsubsection{Proof of Theorem \ref{thm:quantile}}
To prove Theorem \ref{thm:quantile}, we first consider the inner problem of (\ref{OP}), that is,
\begin{align}\label{OP-1}
\sup_{F\in \mathbb B_p(F_0,\varepsilon)} \E^F \left[ \ell_{\alpha}(Y - \boldsymbol{\beta}^\top \mathbf{X}-s) \right].
\end{align}
By Theorem  5 of  \citet{mao2022model}, we can rewrite the problem of (\ref{OP-1}) as a worst case expected loss problem over an one-dimensional Wasserstein ball  immediately.

\begin{lemma}
    \label{projection}
    For $p\geq 1$, $\varepsilon\ge 0$ and  $F_0\in \mathcal M(\R^d)$, $\boldsymbol{\beta}\in\R^d$ and $s\in\R$, it holds that
    \begin{align}
        \sup_{F\in \mathbb B_p(F_0,\varepsilon)} \E^F \left[ \ell_{\alpha}(Y - \boldsymbol{\beta}^\top \mathbf{X}-s) \right] = \sup_{ G\in  \mathbb B_p(G_0, \varepsilon\|\overline{\boldsymbol{\beta}}\|_*)} \E^G \left[ \ell_{\alpha}(Z-s) \right],
    \end{align}
    where $G_0$ is the distribution of $Z = Y_0 - \boldsymbol{\beta}^\top \mathbf{X}_0$ with $(\mathbf{X}_0,Y_0)\sim F_0$.
\end{lemma}

 By virtue of Lemma \ref{projection}, this problem admits an one-dimension version prior to its generalization to higher dimensions. The following proposition provides a convex reformulation to the one-dimension problem, which is the key to reformulating the original problem (\ref{OP}).
 \begin{lemma}\label{prop:quantile}
       For $p\geq 1$, $G_0\in\mathcal M(\R)$, and $\varepsilon\ge0$,  we have
    \begin{align}
   \sup_{  G\in  \mathbb B_p(G_0, \varepsilon)} \E^G \left[ \ell_{\alpha}(Z) \right]
       =
       \begin{cases}
       \E^{G_0} \left[\ell_{\alpha}(Z) \right] + \varepsilon \max\{\alpha,1-\alpha\}   ,&p=1,\\
       \inf_{\lambda>0} \left\{ \E^{G_0} \left[\ell_{\alpha}(Z+k_2 \lambda^{1-q})  \right]+\lambda \varepsilon^p + k_1 \lambda^{1-q} \right\}, &p\in(1,\infty),\\
        \E^{G_0}[\ell_\alpha(Z+(2\alpha-1)\varepsilon)] + 2\alpha(1-\alpha)\varepsilon, &p=\infty,
       \end{cases}
    \end{align}
 where
    \begin{align*}
            k_1 &=  \frac{p-1}{p^{q}} \left( \alpha^{q}(1-\alpha) + \alpha(1-\alpha)^{q} \right) ~~{\rm and}~~
            k_2  = \frac{p-1}{p^{q}}\left( \alpha^{q} - (1-\alpha)^{q} \right).
    \end{align*}
 \end{lemma}
\begin{proof}[Proof]
    When \(p=1\), the problem reduces to classic Lipschitz regularization framework, a well-established case covered in \citet{shafieezadeh2019regularization}.

When \(p\in(1,\infty)\), by classic strong duality result of Wasserstein DRO (see, e.g., \citet{gao2023distributionally}), we have
    \begin{align}
    \label{eq:dual-1}
       \sup_{  G\in  \mathbb B_p(G_0, \varepsilon)} \E^{G} \left[\ell_{\alpha}(Z) \right] &= \inf_{\lambda>0} \left\{ \E^{G_0} \left[\phi_\lambda(Z)\right]+\lambda \varepsilon^p \right\},
    \end{align}
    where $\phi_\lambda(z) = \sup_{y\in\R} \left\{\ell_\alpha(y) - \lambda |y-z|^p\right\}$. Define
    \begin{align*}
        g(y) := \ell_\alpha(y) - \lambda |y-z|^p&= \begin{cases}
            \alpha y - \lambda |y-z|^p, & y\geq 0,\\
            (\alpha-1)y - \lambda |y-z|^p, & y<0.
        \end{cases}
    \end{align*}
    Note that $g(y)$ is concave and differentiable on both $(-\infty,0)$ and $(0,\infty)$, and that 0 is not a maximum point of $g(y)$. We can thus consider the two cases separately by calculating the first-order condition of $g(y)$ in each case. If $y>0$, we have
    $g^\prime(y) = \alpha - p\lambda|y-z|^{p-1}\times\sgn(y-z)$, which implies an possible maximum point at $y_1 = z + \left({\alpha}/{p\lambda}\right)^{q/p}$. If $y<0$, we can similarly get another possible maximum point at $y_2 = z - \left({1-\alpha})/{p\lambda}\right)^{q/p}$.

    Let $y^*$ denote the global maximum point. We first consider the case where only one of the two points lies in their respective domains. If $z>\left(({1-\alpha})/{p\lambda}\right)^{q/p}$, then $y_1>y_2>0$, implying that $y^*=y_1$. Similarly, if $z<-\left({\alpha}/{p\lambda}\right)^{q/p}$, then $y_2<y_1<0$, implying that $y^*=y_2$. As for the case where both points exist, we need to compare the values of $g(y_1)$ and $g(y_2)$ to determine the global maximum point. Note that
    \begin{align*}
        g(y_1) &= \alpha z + (p-1)\left(\frac\alpha p\right)^{q}\lambda^{1-q},\quad g(y_2) = (\alpha-1)z + (p-1)\left(\frac{1-\alpha}{p}\right)^{q}\lambda^{1-q},
    \end{align*}
    and $g(y_1)\geq g(y_2)$ if and only if $z\ge -k_2\cdot\lambda^{1-q} $. This critical condition aligns with the two previous cases. Therefore, we have
    \begin{align*}
        y^* = \begin{cases}
            y_1, & z\ge -k_2\times\lambda^{1-q},\\
            y_2, & z< -k_2\times\lambda^{1-q},
        \end{cases}
    \end{align*}
    and thus
    \begin{align*}
        \phi_\lambda(z) = g(y^*)
        &= \begin{cases}
            \alpha z + (p-1)\left(\frac\alpha p\right)^{q}\lambda^{1-q}, & z+k_2\lambda^{1-q}\geq 0,\\
            (\alpha-1)z + (p-1)\left(\frac{1-\alpha}{p}\right)^{q}\lambda^{1-q}, & z+k_2\lambda^{1-q}< 0.
        \end{cases}\\
        &=\alpha (z+k_2\lambda^{1-q})_+ + (1-\alpha)(z+k_2\lambda^{1-q})_- + k_1\lambda^{1-q} \\&=\ell_\alpha(z+k_2\lambda^{1-q})+ k_1\lambda^{1-q} .
    \end{align*}
    Substituting $\phi_\lambda(z)$ into the dual problem \eqref{eq:dual-1}  gives the desired result.

    Finally, when $p=\infty$, we have
    \begin{align*}
        \sup_{G\in\mathbb{B}_\infty(G_0,\varepsilon)} \E^G \left[ \ell_{\alpha}(Z) \right] = \sup_{|V|\le \varepsilon}\E^{G_0}\left[\ell_\alpha(Z+V)\right] = \E^{G_0}\left[\sup_{|y|\le \varepsilon}\ell_\alpha(Z+y)\right],
    \end{align*}
    where the second supremum is taken over random variables $V$ such that $|V|\le \epsilon$ almost surely.
    Now we consider the inner supremum. Given $z\in\R$, we have
    \begin{align*}
        \sup_{|y|<\varepsilon}\ell_\alpha(z+y) &= \max\{\alpha(z+\varepsilon)_+, (1-\alpha)(z-\varepsilon)_-\}\\
        &= \begin{cases}
            \alpha(z+\varepsilon), & z+(2\alpha-1)\varepsilon \geq 0,\\
            (1-\alpha)(\varepsilon-z), & z+(2\alpha-1)\varepsilon < 0.
        \end{cases}\\
        &= \alpha(z+(2\alpha-1)\varepsilon)_+ + (1-\alpha)(z+(2\alpha-1)\varepsilon)_-+2\alpha(1-\alpha)\varepsilon.
    \end{align*}
    The first equality holds because the supremum of convex function $\ell_\alpha$ is attained at the boundary points. The second and third equalities follow from straightforward calculations. It follows that
    \begin{align*}
        \sup_{G\in \mathbb B_\infty(G_0,\varepsilon)} \E^G[\ell_\alpha(Z)] = \E^{G_0}[\ell_\alpha(Z+(2\alpha-1)\varepsilon)] + 2\alpha(1-\alpha)\varepsilon.
    \end{align*}
    This completes the proof.
\end{proof}
  Now, we are ready to prove   Theorem \ref{thm:quantile}.
\begin{proof}[Proof of Theorem \ref{thm:quantile}]
   When $p=1$, the result directly follows from Lemmas \ref{projection} and  \ref{prop:quantile}. We next consider the case $p>1$. In this case, by Lemmas \ref{projection} and  \ref{prop:quantile}, the quantile regression problem (\ref{OP}) can be reformulated as
    \begin{align} \label{problem:1}
        &\inf_{\boldsymbol{\beta},s,\lambda>0} \left\{\lambda(\varepsilon\|\overline{\boldsymbol{\beta}}\|_*)^p + k_1\lambda^{1-q} + \E^{F_0} \left[\ell_{\alpha}(Y - \boldsymbol{\beta}^\top \mathbf{X}-s+k_2\lambda^{1-q}) \right]\right\}\notag\\
        &= \inf_{\boldsymbol{\beta},\lambda>0} \left\{\lambda(\varepsilon\|\overline{\boldsymbol{\beta}}\|_*)^p + k_1\lambda^{1-q} + \inf_{s} \E^{F_0} \left[\ell_{\alpha}(Y - \boldsymbol{\beta}^\top \mathbf{X}-s+k_2\lambda^{1-q}) \right]\right\}\notag\\
        & = \inf_{\boldsymbol{\beta}\in\R^d,\lambda>0} \left\{\lambda(\varepsilon\|\overline{\boldsymbol{\beta}}\|_*)^p + k_1\lambda^{1-q} + \inf_{s\in\R} \E^{G_{\boldsymbol{\beta}}} \left[\ell_{\alpha}(Z_\beta-s+k_2\lambda^{1-q}) \right]\right\}\notag\\
         & = \inf_{\boldsymbol{\beta}\in\R^d,\lambda>0} \left\{\lambda(\varepsilon\|\overline{\boldsymbol{\beta}}\|_*)^p + k_1\lambda^{1-q} + \inf_{\overline{s}\in\R} \E^{G_{\boldsymbol{\beta}}} \left[\ell_{\alpha}(Z_\beta-\overline{s}) \right]\right\}\notag\\
         &= \inf_{\boldsymbol{\beta}\in\R^d} \left\{\inf_{\lambda>0}\{\lambda(\varepsilon\|\overline{\boldsymbol{\beta}}\|_*)^p+ k_1\lambda^{1-q} \} + \inf_{\overline{s}\in\R} \E^{G_{\boldsymbol{\beta}}} \left[\ell_{\alpha}(Z_\beta-\overline{s}) \right]\right\},
    \end{align}
    with $k_1$ and $k_2$ defined in Lemma  \ref{prop:quantile}.    The first equality holds because the value ranges of $\overline{\boldsymbol{\beta}}$ and $s$ are irrelevant,
    $G_{\boldsymbol{\beta}}$ is the distribution of $Z_\beta= Y_0 - \boldsymbol{\beta}^\top \mathbf{X}_0$ with $(  \mathbf{X}_0,Y_0 )\sim F_0$, and
    the third equality follows from that $s$ can take any value  in $\R$ and thus,
    \begin{align}\label{eq:temp}
        \inf_{s\in\R} \E^{G_{\boldsymbol{\beta}}} \left[\ell_{\alpha}(Z_\beta-s+k_2\lambda^{1-q}) \right]=\inf_{\overline{s}\in\R} \E^{G_{\boldsymbol{\beta}}} \left[\ell_{\alpha}(Z_\beta-\overline{s}) \right],
    \end{align}  the last equality follows from that $\inf_{\overline{s}\in\R} \E^{G_{\boldsymbol{\beta}}} \left[\ell_{\alpha}(Z_\beta-\overline{s}) \right]$ is independent of $\lambda$.
     Moreover, consider the first minimization problem of problem \eqref{problem:1}, $\inf_{\lambda>0}\: \lambda(\varepsilon\|\overline{\boldsymbol{\beta}}\|_*)^p + k_1 \lambda^{1-q}$.  By standard manipulations, one can verify that the optimal value of $\lambda$ is
       $ \lambda^{*} =   ({(q-1)k_1})^{{1}/{q}}/{(\varepsilon\|\overline{\boldsymbol{\beta}}\|_*)}^{{p}/{q}} $, and thus,
    \begin{align*}
        \inf_{\lambda>0}\: \lambda(\varepsilon\|\overline{\boldsymbol{\beta}}\|_*)^p + k_1 \lambda^{1-q} = \varepsilon\|\overline{\boldsymbol{\beta}}\|_*  k_1^{ {1}/{q}} q(q-1)^{-{1}/{p}}.
    \end{align*}
   Substituting this into the problem (\ref{problem:1}), we have problem \eqref{OP} is equivalent to
     \begin{align*}
        &\inf_{\boldsymbol{\beta}\in\R^d,\overline{s}\in \R} \left\{  \E^{F_0} \left[\ell_{\alpha}(Z_{ \boldsymbol{\beta}}-\overline{s})\right] +\varepsilon\|\overline{\boldsymbol{\beta}}\|_*  k_1^{{1}/{q}} q(q-1)^{-{1}/{p}}\right\}\\
        &=\inf_{\boldsymbol{\beta}\in\R^d,\overline{s}\in \R}  \left\{\E^{F_0} \left[ \ell_{\alpha}(Y - \boldsymbol{\beta}^\top \mathbf{X}-\overline{s}) \right] + C\varepsilon\|\overline{\boldsymbol{\beta}}\|_* \right\} .
    \end{align*}
  That is the problem \eqref{prob-1}. The relation between  the optimal values of $\overline{s}$ to problem \eqref{OP}  and $s$ to  problem \eqref{prob-1} is due to \eqref{eq:temp}. This completes the proof.

\end{proof}

\subsubsection{Proof of Theorem \ref{thm:equivalence}}
The proof of Theorem~\ref{thm:equivalence} uses Proposition~\ref{prop:worst-case-p}, whose proof is deferred to Section~\ref{sec:proof2}.
 To prove the ``only if'' part, we first give the following lemmas. Without loss of generality, we assume that $c=1$ in~\eqref{eq:cond-rho} throughout the rest of this section.

\begin{lemma}\label{lem:ell}
    For $p>1$, let $\ell:\R\to\R$ be a convex function and $\rho:\mathcal M_p(\R)\to \R$ be given by $ \rho^F(X):=\inf_{s\in\R}  \E^F \left[ \ell(X-s) \right]$.
    If \eqref{eq:cond-rho} holds for any $F_0\in\M_p(\R)$ and $\varepsilon\ge0$, then we have the following statements hold.
    \begin{enumerate}[(i)]
        \item The function $\ell$ is coercive, and thus, $\ell$ is decreasing on $(-\infty,\overline{x}]$ and increasing on $[\overline{x},\infty)$ for some $\overline{x}\in\R$.   \label{item:lem-ell-i}
        \item For any $F\in\mathcal M_p(\R)$, $\inf_{x\in\R}\E^{F} \left[ \ell(X-x) \right]$ is attained by some finite $x^*\in [\essinf^F(X)-\overline{x} , \esssup^F(X)-\overline{x}]$.    \label{item:lem-ell-ii}
        \item For $r\in (1,p)$,  there exists $M>0$ and such that $\ell(x)\le M(1+|x|^r)$ for all $x\in\R$. \label{item:lem-ell-iii}
    \end{enumerate}
\end{lemma}

\begin{proof}[Proof]
  (i)  We  prove it by contradiction. Assuming that $\ell$ is not coercive, by the convexity of $\ell$, we have $\ell$ is monotone,  and thus, for any  $F$ with bounded support
  , we have $\rho^F(X) =\inf_{x\in\R}\E^{F} \left[ \ell(X-x) \right] =\inf_{x\in\R}\ell(x)$. This contradicts $\rho^F(X)\in\R$ for any $F$ with bounded support if $\inf_{x\in\R}\ell(x)=-\infty$.
  Now assume $\inf_{x \in \mathbb{R}} \ell(x) = L > -\infty$ and, without loss of generality, let $\ell$ be non-decreasing. For any $F \in \mathcal{M}_p(\mathbb{R})$ and $X\sim F$, since $\rho^F(X)\in\R$, there exists $x_0 \in \mathbb{R}$ such that $\mathbb{E}^F[\ell(X-x_0)] < \infty$. For all $x > x_0$, the monotonicity of $\ell$ implies that $\ell(X-x) \le \ell(X-x_0)$, and $\ell(X-x) \downarrow L$ almost surely as $x \to \infty$.
Since $\ell(X-x)$ is bounded below by $L$ and above by $\ell(X-x_0)$, the Dominated Convergence Theorem yields $\rho^F(X) = \inf_{x \in \mathbb{R}} \mathbb{E}^F[\ell(X-x)] = \lim_{x \to \infty} \mathbb{E}^F[\ell(X-x)] = \mathbb{E}^F \left[ \lim_{x \to \infty} \ell(X-x) \right] = L$ for any $F \in \mathcal{M}_p(\mathbb{R})$,
which contradicts \eqref{eq:cond-rho}.
  Thus, we complete the proof of (i).

    (ii) It suffices to consider the case $\ell(0)=\min_{x}\ell(x)$, and we aim to show there exists a minimizer $x^*$ such that $\essinf(X)\le x^* \le \esssup(X)$.
    The existence of $x^*$ follows immediately from   $\lim_{x\to\pm\infty}\E^{F} \left[ \ell(X-x)\right]=\infty$ by (i) and the convexity of $x\mapsto\E^F[\ell(X-x)]$ for any $F\in \mathcal M_p(\R)$.
    If $\essinf(X)=-\infty$ and $\esssup(X)=\infty$, then the claim holds trivially. If $\essinf(X)>-\infty$, then $\E[\ell(X-x)]$ is non-increasing for $x<\essinf(X)$ which implies that there exists a minimizer $x^*$ such that $x^*\ge \essinf(X)$. Similarly, if $\esssup(X)<\infty$, then there exists a minimizer $x^*$ such that $x^*\le \esssup(X)$. Thus, we complete the proof of (ii).

    (iii)  We show the result by contradiction.   Given $r\in(1,p)$, assume the contrary, i.e., for any $n\in\N$, there exists $x_n$ such that $\ell(x_n)> n(1+|x_n|^r)$, that is,
    \begin{equation} \label{eq:260218-1}
        \lim_{n\to\infty} \frac{\ell(x_n)}{1+|x_n|^r} = \infty.
    \end{equation}
    Since $\ell(x)/(1+|x|^r)$ is continuous, for any bounded sequence, the image is bounded. Thus $\{x_n\}_{n\in\N}$ must be unbounded. Without loss of generality, assume $x_n>0$ and $x_n\to\infty$.  Note that (i) implies that $\ell$ admits at least one minimizer. Thus, by shifting $\ell$, we may assume that $\ell$ attains its minimum at $0$ and $\ell(0)=0$; moreover, $\ell(x)=\ell_1(x_+)+\ell_2(x_-)$ for some non-decreasing and non-constant functions $\ell_1,\ell_2:\R_+\to\R_+$ and $\ell_1(0)=\ell_2(0)=0$. By the convexity and monotonicity of $\ell_2$ which is non-constant,
    there exist $\kappa>0$ and $x_0\ge 0$ such that $\ell_2(x)\ge \kappa (x-x_0)$ for all $x\ge 0$. Assume $\kappa<1$ without loss of generality.

  Take $F_0=\delta_0$, and for  any $y>0$, denote by $\epsilon_y\in(\frac{\kappa}{4}y,\frac{\kappa}{3}y)$. Since $\kappa<1$, we have $\epsilon_y < y$. Consider a two-point distribution $F_y=(1-\pi)\delta_0 +\pi\delta_y$ with  $\pi = (\epsilon_y/y)^p$. We have
    \begin{align}
        \rho^{F_y}(X) &= \inf_{x\in[0,y]} \pi \ell_1((y - x)_+)+\left(1-\pi\right) \ell_2((-x)_-)\notag \\
        &\ge \inf_{x\in[0,y]} \left(\frac{\kappa}{4}\right)^p \ell_1(y - x)+\left(1-\left(\frac{\kappa}{3}\right)^p\right) \ell_2(x)\notag\\
        &> \inf_{x\in[0,y]} \left(\frac{\kappa}{4}\right)^p \ell_1(y - x)+\frac{1}{2} \kappa (x-x_0)\label{eqlast-1}\\
        &=  \left(\frac{\kappa}{4}\right)^p \ell_1(y - x_y^*)+\frac{1}{2} \kappa (x_y^*-x_0), \label{eqlast}
    \end{align}
    where the first equality we used the minimizer of $x$ lies in $[0,y]$,   the last inequality follows from $\ell_2(x)\ge \kappa (x-x_0)$ and $1-(\kappa/3)^p > 1/2$ as $\kappa<1$, and the last equality follows by denoting by $x_y^*\in [0,y]$ the minimizer to the problem \eqref{eqlast-1}.
    Since $\ell(0)=0$ implies $\rho^{F_0}(X)=0$,  it follows from  \eqref{eq:cond-rho} and $W_p(F_0,F_y)= (\pi y^p)^{1/p} = \epsilon_y$ that
    \begin{equation}\label{eq:rhoFy_bound}
        \rho^{F_y}(X) \le \rho^{F_0}(X) + \epsilon_y = \epsilon_y < \frac{\kappa}{3} y.
    \end{equation}
    Combining \eqref{eq:rhoFy_bound} with \eqref{eqlast} and noting $\ell_1\ge 0$, we have
    $
         \kappa(x_y^*-x_0)/2 <  {\kappa y}/{3} $, that is,  $x_y^* < {2}  y/3 + x_0.
    $
    Now, letting $y = 3(x_n + x_0)$,  we have
   $
        y - x_y^* > y - \left({2}  y/3 + x_0\right)
        = x_n.
     $
    It then follows from \eqref{eqlast} that  for sufficiently large $n$, it holds that
    \begin{align*}
        \rho^{F_y}(X) &> \left(\frac{\kappa}{4}\right)^p \ell_1(y - x_y^*)+\frac{1}{2}\kappa(x_y^*-x_0)
        > \left(\frac{\kappa}{4}\right)^p \ell_1(x_n) -\frac{1}{2}\kappa x_0> \left(\frac{\kappa}{4}\right)^p n (1+|x_n|^r)-\frac{1}{2}\kappa x_0,
    \end{align*}
    where the second inequality follows from the monotonicity of $\ell_1$ and $x_y^*\ge 0$, and the third inequality follows from the assumption \eqref{eq:260218-1}.
    On the other hand, \eqref{eq:rhoFy_bound} implies $\rho^{F_y}(X) < \frac{\kappa}{3}y = \kappa(x_n+x_0)$.
    Thus, we have
    \begin{equation*}
        \left(\frac{\kappa}{4}\right)^p n (1+|x_n|^r)-\frac{1}{2}\kappa x_0 < \kappa(x_n+x_0).
    \end{equation*}
    Since $r>1$, we have $|x_n|^{r}/|x_n| \to \infty$ as $n\to\infty$.
Hence the left-hand side grows faster than linearly in $x_n$,
whereas the right-hand side is linear in $x_n$,
leading to a contradiction as $n \to \infty$. This completes the proof.
\end{proof}

\begin{lemma}\label{lem:rho}

     For $p>1$, let $\ell:\R\to\R$ be a convex function and $\rho:\mathcal M_p(\R)\to \R$ be given by $ \rho^F(X):=\inf_{s\in\R}  \E^F \left[ \ell(X-s) \right]$. If \eqref{eq:cond-rho} holds for any $F_0\in\M_p(\R)$ and $\varepsilon\ge0$, then the map $X\mapsto \rho^F(X)$ is convex, and the map $F\mapsto \rho^{F}(X)$ is weakly upper-semicontinuous on $\mathbb{B}_p(F_0,\varepsilon)$ for any $F_0\in\mathcal M_p(\R)$ and $\varepsilon\ge0$.
\end{lemma}
\begin{proof}[Proof] 
    The convexity of $X\mapsto \rho^{F}(X)$ is straightforward to check.

    The weak upper-semicontinuity of $F\mapsto \rho^{F}(X)$ on $\mathbb{B}_p(F_0,\varepsilon)$ follows from the weak upper-semicontinuity of $F\mapsto \E^F[\ell(X-x)]$ for any fixed $x\in\R$, which is guaranteed by Lemma~\ref{lem:ell}(\ref{item:lem-ell-iii}) and the proof of Theorem~3 in \cite{yue2022linear}. In fact, for any sequence $\{F_n\}\subseteq \mathbb{B}_p(F_0,\varepsilon)$ weakly converging to some $F\in \mathbb{B}_p(F_0,\varepsilon)$, we have
        \begin{align*}
            \limsup_{n\to\infty} \rho^{F_n}(X) &= \limsup_{n\to\infty} \inf_{x\in\R}  \E^{F_n} \left[ \ell(X-x) \right] \\
            &\le \inf_{x\in\R} \limsup_{n\to\infty}  \E^{F_n} \left[ \ell(X-x) \right] \le \inf_{x\in\R}  \E^{F} \left[ \ell(X-x) \right] = \rho^{F}(X).
        \end{align*}
        This completes the proof.
\end{proof}

We are now ready to prove Theorem~\ref{thm:equivalence}.
\begin{proof}[Proof of Theorem~\ref{thm:equivalence}]
    We only need to prove the ``only if'' part. Without loss of generality, assume that $\ell(0)=\min_x\ell(x) =0$ by Lemma \ref{lem:ell}. We can write $\ell(x)=\ell_1(x_+)+\ell_2(x_-)$.
    First, note that the Wasserstein ball $\mathbb{B}_p(F_0,\varepsilon)$ is weakly compact for any $F_0\in\mathcal M_p(\R)$ and $\varepsilon\ge0$ \citep[Theorem~1]{yue2022linear}. Thus, by Lemma~\ref{lem:rho}, the Weierstrass theorem applies and the supremum in $\sup_{F\in \mathbb{B}_p(F_0,\varepsilon)}  \rho^F(X)$ is attained for any $F_0\in\mathcal M_p(\R)$ and $\varepsilon\ge0$. Following the proof of Lemma~\ref{lem:ell}, let $F_0=\delta_0.$ Denote the worst-case distribution attaining the supremum in $\sup_{F\in \mathbb{B}_p(\delta_0,\varepsilon)}  \rho^F(X)$ by $F^*_\varepsilon$ and let $X^*_\varepsilon\sim F^*_\varepsilon$. For notational convenience, we write $\rho(X^*_\varepsilon)$ in place of $\rho^{F^*_\varepsilon}(X)$ to allow direct manipulations of the random variable.  Denote by $x_{\sup}=\esssup(X_\varepsilon^*)$  and $x_{\inf}=\essinf(X_\varepsilon^*)$. We have $x_{\inf}<0<x_{\sup}$, as $0$ minimizes $\inf_{t\in\R} \E[|X^*_\epsilon - t|^p] $. To see it,  note that if $\epsilon
    =  W_p(X^*_\epsilon,\delta_0)=\left(\E[|X^*_\epsilon |^p] \right)^{1/p}>\left(\E[|X^*_\epsilon - t|^p]\right)^{1/p}=W_p(X^*_\epsilon,\delta_t)=:\epsilon'$  for $t\neq 0$, then by \eqref{eq:cond-rho}, we have $ \rho(X^*_\varepsilon) \le \rho(t) + \varepsilon' = \varepsilon'<\epsilon= \rho(X^*_\varepsilon) $ yielding a contradiction.

    By \eqref{eq:cond-rho}, we have $ \rho(X^*_\varepsilon) = \rho(0) + \varepsilon = \varepsilon$, and $\E[|X^*_\varepsilon|^p] \le \varepsilon^p$. We claim that
    \begin{equation}\label{eq:linear-rho}
        \rho(\lambda X^*_\varepsilon) = \lambda \rho(X^*_\varepsilon)=\lambda \varepsilon \text{ for any } \lambda>0.
    \end{equation}
    We show \eqref{eq:linear-rho} by contradiction.
    Suppose   there exists  $\lambda_0>0$ such that $\rho(\lambda_0 X^*_\varepsilon) < \lambda_0 \varepsilon$. Consider first the case $\lambda_0\in(0,1)$. Denote by $F_{\lambda_0}$ the distribution of $\lambda_0 X^*_\varepsilon$. We have $\E[|X^*_\varepsilon - \lambda_0 X^*_\varepsilon|^p]^{1/p} \le (1-\lambda_0)\varepsilon$ which implies $F^*_\varepsilon\in \mathbb{B}_p(F_{\lambda_0},(1-\lambda_0)\varepsilon) $, and thus,
    \begin{equation*}
        \rho(X_\varepsilon^*)\le \sup_{F\in \mathbb{B}_p(F_{\lambda_0},(1-\lambda_0)\varepsilon)}  \rho^F(X) = \rho(\lambda_0 X^*_\varepsilon) + (1-\lambda_0)\varepsilon < \varepsilon,
    \end{equation*}
    where the equality follows from    \eqref{eq:cond-rho}. This  contradicts  $\rho(X_\varepsilon^*) = \varepsilon$. Next, consider the case $\lambda_0>1$. It follows from the convexity of $\rho(X)$ by lemma \ref{lem:rho} that
    \begin{equation*}
        \rho(X^*_\varepsilon) = \rho\left(\frac{1}{\lambda_0} \lambda_0 X^*_\varepsilon + \left(1-\frac{1}{\lambda_0}\right) 0\right) \le \frac{1}{\lambda_0} \rho(\lambda_0 X^*_\varepsilon) + \left(1-\frac{1}{\lambda_0}\right) \rho(0)< \varepsilon,
    \end{equation*}
    which again leads to a contradiction. Thus, the claim \eqref{eq:linear-rho} holds.

    Next, we show that $\ell_1$ and $\ell_2$ are  linear on $\R_+$. Denote by  $x^*$ a minimizer of $\inf_{x\in\R}\E[\ell(X_\varepsilon^*-x)]$ such that $x_{\inf}\le x_\varepsilon^* \le x_{\sup}$, which exists by Lemma~\ref{lem:ell}(\ref{item:lem-ell-ii}). We have $\rho(X^*_\epsilon)
     = \E \left[ \ell (X^*_\epsilon - x^*)  \right]= \varepsilon$, and thus by \eqref{eq:linear-rho}, for any $\lambda\in(0,1)$, it holds that
\begin{align*}
    \lambda \varepsilon = \rho(\lambda X^*_\epsilon) & = \inf_{x\in\R} \E \left[ \ell(\lambda X^*_\epsilon - x) \right]\\
    & \le \E \left[ \ell(\lambda X^*_\epsilon - \lambda x^*) \right] \\
    &= \E\left[ \ell_1(\lambda (X^*_\epsilon - x^*)_+) \right] + \E\left[ \ell_2(\lambda (X^*_\epsilon - x^*)_-) \right]\\
    &\le \lambda \E\left[ \ell_1((X^*_\epsilon - x^*)_+) \right] + \lambda \E\left[ \ell_2((X^*_\epsilon - x^*)_-) \right] = \lambda \varepsilon,
\end{align*}
where  the last inequality holds  due to $\ell_i(\lambda x) \le \lambda \ell_i(x) $ for $x\ge 0$ and $i=1,2$ by the convexity of $\ell_1$ and $\ell_2$ and  $\ell_1(0)=\ell_2(0)=0$. Thus, all inequalities above are equalities. In particular,  by $\ell_i(\lambda x) \le \lambda \ell_i(x) $ for $x\ge 0$ and $i=1,2$  again, we have
\begin{equation}\label{eq:260208-2}
     \ell_1(\lambda (X^*_\epsilon - x^*)_+)   = \lambda \ell_1((X^*_\epsilon - x^*)_+)  ,~~
       \ell_2(\lambda (X^*_\epsilon - x^*)_-)  = \lambda   \ell_2((X^*_\epsilon - x^*)_-)  ~{\rm a.s.},\quad \forall \lambda\in[0,1].
\end{equation}
Since $\rho^F(X)=0$ for any degenerate distribution $F$, we have $x_{\sup}>x_{\inf}$ and thus at least one of the two intervals $[0,x_{\sup} - x^*)$ and $[0,x^* - x_{\inf})$ is non-trivial.  We consider the following three cases.
\begin{itemize}
    \item [1.] If $x_{\inf}<x^*<x_{\sup}$, then   \eqref{eq:260208-2} implies  $\ell_1$ is linear on $[0,x_{\sup} - x^*)$  and $\ell_2$ is linear on $[0,   x^*-x_{\inf})$. Furthermore, by \eqref{eq:linear-rho}, we have for any $\lambda>0$, the distribution of $\lambda X^*_\epsilon$ is also a worst-case distribution and $\lambda x^*$ is a minimizer of $\inf_{x\in\R}\E[\ell(\lambda X_\epsilon^*-x)]$. By selecting $\lambda$ sufficiently large, we can extend the non-trivial linearity interval to $\R_+$, i.e., both $\ell_1$ and $\ell_2$ are linear on $\R_+$.

   \item [2.]  If $x_{\sup}=x^*>x_{\inf}$,    then \eqref{eq:260208-2} implies   $\ell_2$ is linear on $[0,   x^*-x_{\inf})$.
   Similar to Case 1, we can show $\ell_2$ is linear on $\R_+$,

that is, $\ell_2(x)=b x$ for some $b>0$ on $\R_+$. We next prove that $\ell_1$ is also linear on $\R_+$ by contradiction. Suppose that $\ell_1$ is not linear on $\R_+$. The first order condition of the minimization problem $\inf_{x\in\R}\E[\ell(X_\epsilon^*-x)]$ gives (see, e.g., \cite[Proposition~1]{bellini2014generalized})
\begin{equation*}
    \E \left[ \ell_{1+}^{\prime}((X^*_\epsilon - x_{\sup})_+) \mathbbm{1}_{\{X^*_\epsilon \ge x_{\sup}\}} \right] \ge \E \left[ \ell_{2-}^{\prime}((X^*_\epsilon - x_{\sup})_-) \mathbbm{1}_{\{X^*_\epsilon < x_{\sup}\}} \right],
\end{equation*}
where $\ell_{1+}^{\prime}$ and $\ell_{2-}^{\prime}$ denote the right derivative of $\ell_1$ and the left derivative of $\ell_2$, respectively. Denoting $\pi = \P(X^*_\epsilon = x_{\sup})$, we have
\begin{equation}
    \pi \ell_{1+}^{\prime}(0) \ge (1-\pi) b. \label{eq:first-order}
\end{equation}
The inequality is in fact an equality, since otherwise, we can transform a small mass from $x_{\sup}$ to $0$ such that \eqref{eq:first-order} still holds, implying that $x^*$ remains a minimizer and the objective value $\rho(X)$ non-decreases, while the $p$-th moment decreases, contradicting \eqref{eq:cond-rho}. Specifically, noting \eqref{eq:first-order} implies $\pi>0$, define $G=(\pi'-\pi)\delta_{x_{\sup}} + (\pi-\pi')\delta_0 + F^*_\epsilon$ for $\pi'<\pi$ such that  $\pi' \ell_{1+}^{\prime}(0) \ge (1-\pi') b.$ This implies   $$\rho^G(X)=\inf_{x}\E^G[\ell(X-x)]=\E^G[\ell(X-x^*)] \ge \E^{F^*_\epsilon}[\ell(X-x^*)]=\rho^{F^*_\epsilon}(X),$$ where the inequality follows from  $x_{\sup}>0$.  Moreover, noting that $\epsilon':=W_p(G,F_0) < W_p( F^*_\epsilon,F_0)$, by \eqref{eq:cond-rho}, we have $\rho^G(X) \le \rho^{F_0}(X) +\epsilon' <\epsilon$, which yields a  contradiction.

Therefore, we have
\begin{equation}
    \pi \ell_{1+}^{\prime}(0) = (1-\pi) b. \label{eq:f.o.c-equality}
\end{equation}
Define $a:= \ell_{1+}^{\prime}(0)>0$ and $\tilde{\ell}(x) := a x_+ + b x_-$, $x\in\R$. By the convexity of $\ell_1$ and $\ell_1(0)=0$, we have  $\ell(x) \ge \tilde{\ell}(x)$ on $\R$ and $\ell(x) = \tilde{\ell}(x)$ on $(-\infty,0]$. We claim that
\begin{equation}\label{eq:260208-3}
    \sup_{F\in \mathbb{B}_p(\delta_0,\varepsilon)} \inf_{x\in\R} \E \left[ \tilde{\ell}(X - x) \right] = \sup_{F\in \mathbb{B}_p(\delta_0,\varepsilon)} \inf_{x\in\R} \E \left[ \ell(X - x) \right] = \varepsilon.
\end{equation}
To see it, it suffices to show \begin{equation*}
    \sup_{F\in \mathbb{B}_p(\delta_0,\varepsilon)} \inf_{x\in\R} \E \left[ \tilde{\ell}(X - x) \right] \ge \sup_{F\in \mathbb{B}_p(\delta_0,\varepsilon)} \inf_{x\in\R} \E \left[ \ell(X - x) \right] = \varepsilon.
\end{equation*}
Note that $F^*_\epsilon\in \mathbb{B}_p(\delta_0,\varepsilon)$ is a feasible distribution for the left-hand side  problem and
\eqref{eq:f.o.c-equality} is also the first order condition of $\inf_{x\in\R} \E \left[ \tilde{\ell}(X^*_\epsilon - x) \right]$. Thus, $x_{\sup}$ is also a minimizer of $\inf_{x\in\R} \E \left[ \tilde{\ell}(X^*_\epsilon - x) \right]$. We have
\begin{equation*}
    \sup_{F\in \mathbb{B}_p(\delta_0,\varepsilon)} \inf_{x\in\R} \E \left[ \tilde{\ell}(X - x) \right] \ge \inf_{x\in\R} \E \left[ \tilde{\ell}(X^*_\epsilon - x) \right] = \E \left[ \tilde{\ell}(X^*_\epsilon - x_{\sup}) \right] = \E \left[ \ell(X^*_\epsilon - x_{\sup}) \right] = \varepsilon,
\end{equation*}
where the second equality follows from  that $X^*_\epsilon - x_{\sup} \le 0$ almost surely. This proves \eqref{eq:260208-3}.

Finally, we show that $\ell(x) = \tilde{\ell}(x)$ for all $x\in\R$.
Assume the contrary, i.e., there exists some $x_0>0$ such that $\ell(x_0) > \tilde{\ell}(x_0)$. Due to the convexity of $\ell_1$ and linearity of $\tilde{\ell}$ on $\R_+$, we have  $\ell(x) > \tilde{\ell}(x)$ for any $x>x_0$.

We now apply some results on the problem with loss function $\tilde{\ell}$ established in our previous sections.
Applying the minmax theorem (see, e.g., \cite[Theorem~2.6]{shafieezadeh2023new}), we have
\begin{equation} \label{eq:minmax}
    \inf_{x\in\R} \sup_{F\in \mathbb{B}_p(\delta_0,\varepsilon)} \E^F \left[ \tilde{\ell}(X - x) \right] = \sup_{F\in \mathbb{B}_p(\delta_0,\varepsilon)} \inf_{x\in\R} \E^F \left[ \tilde{\ell}(X - x) \right] = \varepsilon,
\end{equation}
where the second equality follows from \eqref{eq:260208-3}, and the same argument is valid for $\ell$ as well.
By Proposition~\ref{prop:worst-case-p}, there exist  a two-point distribution $\tilde{F}^*_{\epsilon} = \alpha \delta_{\tilde{x}_1 } + (1-\alpha)\delta_{\tilde{x}_2}$ with $\alpha=a/(a+b)$, and  $\tilde{x}^*={\varepsilon}\left(\alpha^{q}-(1-\alpha)^{q}\right) c_{\alpha,p}^{1-q}/{q} \in (\tilde{x}_2,\tilde{x}_1)$   such that

\begin{equation*}
     \E^{\tilde{F}^*_{\epsilon}} \left[ \tilde{\ell}(\tilde{X}^*_{\epsilon} - \tilde{x}^*) \right] = \sup_{F\in \mathbb{B}_p(\delta_0,\varepsilon)}  \E^F \left[ \tilde{\ell}(X - \tilde{x}^*) \right]  = \varepsilon,
\end{equation*}
where $\tilde{X}^*_{\epsilon}\sim \tilde{F}^*_{\epsilon}$. Note that any point in $(\tilde{x}_2,\tilde{x}_1)$ including $x^*$ is an $\alpha$-quantile of $\tilde{F}^*_{\epsilon}$, and thus a minimizer of $\inf_{x\in\R} \E^{\tilde{F}^*_{\epsilon}} \left[ \tilde{\ell}(\tilde{X}^*_{\epsilon} - x) \right]$. We then have
\begin{equation*}
    \inf_{x\in\R} \E^{\tilde{F}^*_{\epsilon}} \left[ \tilde{\ell}(\tilde{X}^*_{\epsilon} - x) \right] = \E^{\tilde{F}^*_{\epsilon}} \left[ \tilde{\ell}(\tilde{X}^*_{\epsilon} - \tilde{x}^*) \right] = \varepsilon.
\end{equation*}

When the Wasserstein radius $\varepsilon$ is sufficiently large, we have $\tilde{x}_1-\tilde{x}^*> x_0$. Thus, by $\ell\ge\tilde{\ell}$ and $\ell>\tilde{\ell}$ for $x>x_0$, we have $\E^{\tilde{F}^*_{\epsilon}} \left[ {\ell}(\tilde{X}^*_{\epsilon} - \tilde{x}^*) \right] > \E^{\tilde{F}^*_{\epsilon}} \left[ \tilde{\ell}(\tilde{X}^*_{\epsilon} - \tilde{x}^*) \right] = \varepsilon,$ which gives
\begin{equation} \label{eq:final-contradiction-1}
    \sup_{F\in \mathbb{B}_p(\delta_0,\varepsilon)} \E^F \left[ {\ell}(X - \tilde{x}^*) \right] > \varepsilon .
\end{equation}
By Lemma~\ref{prop:quantile} and the proof of Theorem~\ref{thm:quantile}, we know that $\tilde{x}^*$ is the unique minimizer to the left-hand side problem of \eqref{eq:minmax},   and thus, for any $x\ne \tilde{x}^*$, we have
\begin{equation} \label{eq:final-contradiction-2}
    \sup_{F\in \mathbb{B}_p(\delta_0,\varepsilon)} \E^F \left[ {\ell}(X - x) \right] \ge \sup_{F\in \mathbb{B}_p(\delta_0,\varepsilon)} \E^F \left[ \tilde{\ell}(X - x) \right] > \varepsilon,
\end{equation}
where the first inequality follows from the $\ell\ge \tilde{\ell}$. Following the proof of Lemma~\ref{lem:ell} (\ref{item:lem-ell-ii}), one can show that the infimum of $\inf_{x\in\R} \sup_{F\in \mathbb{B}_p(\delta_0,\varepsilon)} \E^F \left[ {\ell}(X - x) \right]$ is attained, as the inner supremum inherits the convexity, continuity, and coercivity of $\ell$. Therefore, combining this with \eqref{eq:final-contradiction-1} and \eqref{eq:final-contradiction-2}, we have
\begin{equation*}
    \sup_{F\in \mathbb{B}_p(\delta_0,\varepsilon)} \rho^F(X) =
    \sup_{F\in \mathbb{B}_p(\delta_0,\varepsilon)} \inf_{x\in\R} \E^F \left[ {\ell}(X - x) \right] = \inf_{x\in\R} \sup_{F\in \mathbb{B}_p(\delta_0,\varepsilon)} \E^F \left[ {\ell}(X - x) \right] > \varepsilon,
\end{equation*}
a contradiction.
    \item [3.] If $x_{\sup}>x=x_{\inf}$, this case is similar to Case 2.

\end{itemize}
Combining the above three cases, we complete the proof.
\end{proof}

\subsubsection{Proofs of Section~\ref{sec:worst-case}} \label{sec:proof2}

\begin{proof}[Proof of Proposition~\ref{prop:worst-case-p1}]
        First we verify that $F^*\in \mathbb{B}_1(F_0,\varepsilon)$. Let $\pi$ be the joint distribution of $\left((\mathbf{X}_0,Y_0),(\mathbf{X}^*,Y^*)\right)$. Then, $\pi\in \Pi(F_0,F^*)$ and
    \begin{align*}
        \E^{\pi}\left[\|(\mathbf{X}_0,Y_0)-(\mathbf{X}^*,Y^*)\|\right] 
        &= \E^{F_0}\left[\left\|\frac{\varepsilon\mathbbm{1}_{\{Y_0 - \boldsymbol{\beta}^\top \mathbf{X}_0 - s \ge 0 \}}}{\P\left(Y_0 - \boldsymbol{\beta}^\top \mathbf{X}_0 - s \ge 0 \right)}\mathbf{v}_{\boldsymbol{\beta}   }\right\|\right] \\
        &= \E^{F_0}\left[\frac{\varepsilon\mathbbm{1}_{\{Y_0 - \boldsymbol{\beta}^\top \mathbf{X}_0 - s \ge 0 \}}}{\P\left(Y_0 - \boldsymbol{\beta}^\top \mathbf{X}_0 - s \ge 0 \right)}\right] \|  \mathbf{v}_{\boldsymbol{\beta}   } \| = \varepsilon.
    \end{align*}
    Thus, by the definition of Wasserstein distance, $W_1(F_0,F^*)\le \varepsilon$, which implies that $F^*\in \mathbb{B}_1(F_0,\varepsilon)$. 
    Next, we show that $F^*$ attains the optimal value of the problem \eqref{sup}. Lemmas~\ref{projection} and \ref{prop:quantile} yield that for $\alpha\geq1/2$
    \begin{align*}
        \sup_{F\in \mathbb{B}_1(F_0,\varepsilon)} 
        \E^F \left[ \ell_{\alpha}(Y - \boldsymbol{\beta}^\top \mathbf{X}-s) \right] 
        &= \E^{F_0} \left[ \ell_{\alpha}(Y - \boldsymbol{\beta}^\top \mathbf{X}-s) \right] + \alpha\varepsilon\|\overline{\boldsymbol{\beta}}\|_*.
    \end{align*} 
    Let $Z^* = Y^* - \boldsymbol{\beta}^\top \mathbf{X}^* - s$ and $Z_0 = Y_0 - \boldsymbol{\beta}^\top \mathbf{X}_0 - s$.
    By the construction of $(\mathbf{X}^*,Y^*)$, we have
    \begin{align*}
        Z^* &= (-\boldsymbol{\beta},1)^\top \left(\mathbf{X}^*,Y^*\right) - s = (-\boldsymbol{\beta},1)^\top \left(\mathbf{X}_0,Y_0\right) + \frac{\varepsilon\mathbbm{1}_{\{Z_0 \ge 0 \}}}{\P\left(Z_0 \ge 0 \right)} (-\boldsymbol{\beta},1)^\top\mathbf{v}_{\boldsymbol{\beta}}  - s = Z_0 + \frac{\varepsilon\mathbbm{1}_{\{Z_0 \ge 0 \}}}{\P\left(Z_0 \ge 0 \right)}\|\overline{\boldsymbol{\beta}}\|_*,
    \end{align*}
    where the last equality follows from the definition of $\mathbf{v}_{\boldsymbol{\beta}}.$
    Note that $Z^*_- = {Z_0}_-$, and $Z^*\ge 0$ if and only if $Z_0\ge 0$. It follows that
    \begin{align*}
        \E^{F^*} \left[ \ell_{\alpha}(Y - \boldsymbol{\beta}^\top \mathbf{X}-s) \right] 
        &= \E \left[ \alpha Z^*_+ + (1-\alpha)Z^*_- \right]\\
        &= \E \left[\alpha Z^*\mathbbm{1}_{\{Z^* \ge 0\}} + (1-\alpha)Z^*_- \right]\\
        &= \E \left[\alpha Z_0\mathbbm{1}_{\{Z_0 \ge 0\}} + \frac{\alpha\varepsilon\mathbbm{1}^2_{\{Z_0 \ge 0 \}}}{\P\left(Z_0 \ge 0 \right)} 
        \|\overline{\boldsymbol{\beta}}\|_* + (1-\alpha){Z_0}_- \right]\\
        &= \E \left[\alpha {Z_0}_+ + (1-\alpha){Z_0}_- \right] + \alpha \varepsilon \|\overline{\boldsymbol{\beta}}\|_* \\
        &= \E^{F_0} \left[ \ell_{\alpha}(Y - \boldsymbol{\beta}^\top \mathbf{X}-s) \right] + \alpha\varepsilon\|\overline{\boldsymbol{\beta}}\|_*,
    \end{align*}
    where the fourth equality holds due to the fact that $\E  [\mathbbm{1}^2_{\{Z_0 \ge 0 \}} ] = \P\left(Z_0 \ge 0 \right)$.

    Now consider the case that $\p(Z_0 \ge 0) = 0$. For any $F\in \mathbb{B}_1(F_0,\varepsilon)$, define $Z = Y - \boldsymbol{\beta}^\top \mathbf{X}-s$ with $(\mathbf{X},Y)\sim F$. We have $\E^\pi|Z-Z_0| \le \|\overline{\boldsymbol{\beta}}\|_* \E^\pi[\|(\mathbf{X},Y) - (\mathbf{X}_0,Y_0)\|] \le \varepsilon \|\overline{\boldsymbol{\beta}}\|_*$, where $\pi$ is the optimal transport plan between $F_0$ and $F$. When $Z\le 0$, we have $\ell_\alpha(Z) - \ell_\alpha(Z_0) \le (1-\alpha)|Z-Z_0| < \alpha|Z-Z_0|$; when $Z>0$, we have $\ell_\alpha(Z) - \ell_\alpha(Z_0) = \alpha Z - (\alpha-1)Z_0 < \alpha|Z-Z_0|$. We   thus conclude that $\ell_\alpha(Z) - \ell_\alpha(Z_0) < \alpha|Z-Z_0|$ holds $\pi$-a.s., which implies $\E^\pi \left[ \ell_\alpha(Z) - \ell_\alpha(Z_0) \right] < \alpha\varepsilon\|\overline{\boldsymbol{\beta}}\|_*$, and thus the supremum in \eqref{sup} is not attained by any distribution in $\mathbb{B}_1(F_0,\varepsilon)$.
    This completes the proof.
\end{proof}

\begin{proof}[Proof of Proposition~\ref{prop:worst-case-pinfty}]
    First we verify that $F^*\in \mathbb{B}_\infty(F_0,\varepsilon)$. Let $\pi$ be the joint distribution of $\left((\mathbf{X}_0,Y_0),(\mathbf{X}^*,Y^*)\right)$. Then, $\pi\in \Pi(F_0,F^*)$ and
    \begin{align*}
        W_\infty(F_0,F^*)
        &\le \esssup_{\pi}\|(\mathbf{X}_0,Y_0)-(\mathbf{X}^*,Y^*)\| \\
        &= \esssup_{F_0}\left\|\varepsilon\operatorname{sgn}\left(Y_0 - \boldsymbol{\beta}^\top \mathbf{X}_0 - s + (2\alpha-1)\varepsilon\left\|\overline{\boldsymbol{\beta}}\right\|_*\right) \mathbf{v}_{\boldsymbol{\beta}}\right\|\\
        &= \varepsilon.
    \end{align*}

    Next, we show that $F^*$ attains the supremum in \eqref{sup}. Lemma~\ref{projection} and \ref{prop:quantile} yield
    \begin{align*}
        \sup_{F\in \mathbb{B}_\infty(F_0,\varepsilon)}
        \E^F \left[ \ell_{\alpha}(Y - \boldsymbol{\beta}^\top \mathbf{X}-s) \right]
        &= \E^{F_0} \left[ \ell_{\alpha}(Y - \boldsymbol{\beta}^\top \mathbf{X}-s+(2\alpha-1)\varepsilon)\left\|\overline{\boldsymbol{\beta}}\right\|_* \right] + 2\alpha(1-\alpha)\varepsilon\|\overline{\boldsymbol{\beta}}\|_*.
    \end{align*}
    Let $Z^* = Y^* - \boldsymbol{\beta}^\top \mathbf{X}^* - s$ and $Z_0 = Y_0 - \boldsymbol{\beta}^\top \mathbf{X}_0 - s$. By the construction of $(\mathbf{X}^*,Y^*)$, we have
    \begin{align*}
        Z^* &= Z_0 + \varepsilon\operatorname{sgn}\left(Z_0 + (2\alpha-1)\varepsilon\|\overline{\boldsymbol{\beta}}\|_*\right)\|\overline{\boldsymbol{\beta}}\|_*\\
        &=\begin{cases}
        Z_0 + \varepsilon\|\overline{\boldsymbol{\beta}}\|_*, & Z_0 + (2\alpha-1)\varepsilon\|\overline{\boldsymbol{\beta}}\|_*\geq 0,\\
        Z_0 - \varepsilon\|\overline{\boldsymbol{\beta}}\|_*, & Z_0 + (2\alpha-1)\varepsilon\|\overline{\boldsymbol{\beta}}\|_*< 0.
        \end{cases}
    \end{align*}
    From the above expression, we can see that $Z^* $ and $Z_0 + (2\alpha-1)\varepsilon\|\overline{\boldsymbol{\beta}}\|_*$ have the same sign as $2\alpha-1\in(-1,1)$. Thus, we have
    \begin{align*}
        \E^{F^*} \left[ \ell_{\alpha}(Y - \boldsymbol{\beta}^\top \mathbf{X}-s) \right]
        &= \E \left[ \alpha Z^*_+ + (1-\alpha)Z^*_- \right]\\
        &= \E \left[\alpha Z^*\mathbbm{1}_{\{Z^* \ge 0\}} + (\alpha-1)Z^*\mathbbm{1}_{\{Z^* < 0\}} \right]\\
        &= \E \left[\alpha \left(Z_0 + \varepsilon\|\overline{\boldsymbol{\beta}}\|_*\right)\mathbbm{1}_{\{Z_0 + (2\alpha-1)\varepsilon\|\overline{\boldsymbol{\beta}}\|_*\geq 0\}} \right.\\
        &\quad\left. + (\alpha-1)\left(Z_0 - \varepsilon\|\overline{\boldsymbol{\beta}}\|_*\right)\mathbbm{1}_{\{Z_0 + (2\alpha-1)\varepsilon\|\overline{\boldsymbol{\beta}}\|_*< 0\}} \right]\\
        &= \E \left[\alpha \left(Z_0 + (2\alpha-1)\varepsilon\|\overline{\boldsymbol{\beta}}\|_*\right)\mathbbm{1}_{\{Z_0 + (2\alpha-1)\varepsilon\|\overline{\boldsymbol{\beta}}\|_*\geq 0\}} \right.\\
        &\quad + (\alpha-1)\left(Z_0 + (2\alpha-1)\varepsilon\|\left(-\boldsymbol{\beta},1\|_*\right)\|_*\right)\mathbbm{1}_{\{Z_0 + (2\alpha-1)\varepsilon\|\overline{\boldsymbol{\beta}}\|_*< 0\}} \\
        &\left. \quad + 2\alpha(1-\alpha)\varepsilon\|\overline{\boldsymbol{\beta}}\|_*\left(\mathbbm{1}_{\{Z_0 + (2\alpha-1)\varepsilon\|\overline{\boldsymbol{\beta}}\|_*\geq 0\}} + \mathbbm{1}_{\{Z_0 + (2\alpha-1)\varepsilon\|\overline{\boldsymbol{\beta}}\|_*< 0\}}\right) \right]\\
        &= \E \left[\alpha \left(Z_0 + (2\alpha-1)\varepsilon\|\overline{\boldsymbol{\beta}}\|_*\right)_+ + (\alpha-1)\left(Z_0 + (2\alpha-1)\varepsilon\|\overline{\boldsymbol{\beta}}\|_*\right)_- \right] \\
        &\quad + 2\alpha(1-\alpha)\varepsilon\|\overline{\boldsymbol{\beta}}\|_* \\
        &= \E^{F_0} \left[ \ell_{\alpha}(Y - \boldsymbol{\beta}^\top \mathbf{X}-s+(2\alpha-1)\varepsilon)\left\|\overline{\boldsymbol{\beta}}\right\|_* \right] + 2\alpha(1-\alpha)\varepsilon\|\overline{\boldsymbol{\beta}}\|_*.
    \end{align*}
    This completes the proof.
\end{proof}

\begin{proof}[Proof of Proposition~\ref{prop:worst-case-p}]
    First we verify that $F^*\in \mathbb{B}_p(F_0,\varepsilon)$. Let $\pi$ be the joint distribution of $\left((\mathbf{X}_0,Y_0),(\mathbf{X}^*,Y^*)\right)$. Then, $\pi\in \Pi(F_0,F^*)$ and
    \begin{align*}
        \E^{\pi}\left[\|(\mathbf{X}_0,Y_0)-(\mathbf{X}^*,Y^*)\|^p\right] 
        &= \E^{F_0}\left[\left|\varepsilon c_{\alpha,p}^{1-q}\left(\alpha^{q-1}\mathbbm{1}_{A}-(1-\alpha)^{q-1}\mathbbm{1}_{A^c}\right)\right|^p\right] \|\mathbf{v}_{\boldsymbol{\beta}^*} \|^p\\
        &= \varepsilon^p c_{\alpha,p}^{p(1-q)}\E^{F_0}\left[\alpha^{p(q-1)}\mathbbm{1}_{A}+(1-\alpha)^{p(q-1)}\mathbbm{1}_{A^c}\right]\\
        &= \varepsilon^p c_{\alpha,p}^{p(1-q)}\left(\alpha^{p(q-1)}(1-\alpha)+(1-\alpha)^{p(q-1)}\alpha\right)\\
        &= \varepsilon^p c_{\alpha,p}^{-q}\left(\alpha^q(1-\alpha)+(1-\alpha)^q\alpha\right)\\
        &= \varepsilon^p.
    \end{align*}
    The fourth equality uses the fact that $1/p+1/q=1$ and the last equality follows from the definition of $c_{\alpha,p}$ in \eqref{eq:C_alphap}.
    Thus, by the definition of Wasserstein distance, $W_p(F_0,F^*)\le \varepsilon$, which implies that $F^*\in \mathbb{B}_p(F_0,\varepsilon)$.

    Next, we show that $F^*$ attains the supremum in \eqref{sup}. Since $(\boldsymbol{\beta}^*,s^*)$ is an optimal solution to \eqref{OP}, by Theorem~\ref{thm:quantile} we have
    \begin{align*}
        \sup_{F\in \mathbb{B}_p(F_0,\varepsilon)}
        \E^F \left[ \ell_{\alpha}(Y - \boldsymbol{\beta}^{*\top} \mathbf{X}-s^*) \right]
        &= \E^{F_0} \left[ \ell_{\alpha}(Y - \boldsymbol{\beta}^{*\top} \mathbf{X}-\overline{s}^*) \right] + c_{\alpha,p}\varepsilon\|\overline{\boldsymbol{\beta}^*}\|_*,
    \end{align*}
    with $\overline{s}^*$ defined in \eqref{eq:relationbeta0}. Note that $\overline{s}^*$ is in fact the $\alpha$-quantile of $Y_0 - \boldsymbol{\beta}^{*\top} \mathbf{X}_0$ for a fixed $\boldsymbol{\beta}^*$, since it solves
        $\inf_{\overline{s}}  \E \left[ \ell_{\alpha}(Y_0 - \boldsymbol{\beta}^{*\top} \mathbf{X}_0 - \overline{s}) \right].$
    Thus by the proposition's assumption $A\subseteq \{Y_0 - \boldsymbol{\beta}^{*\top} \mathbf{X}_0 \ge \overline{s}^*\}$ and $\P(A)=1-\alpha$, we have
    \begin{align*}
    \left\{Y_0 - \boldsymbol{\beta}^{\top} \mathbf{X}_0 > \overline{s}^*\right\} &\subseteq A \subseteq \left\{Y_0 - \boldsymbol{\beta}^{\top} \mathbf{X}_0 \geq \overline{s}^*\right\}
    \end{align*}
    almost surely.

    Furthermore, we have $(Y_0 - \boldsymbol{\beta}^{*\top} \mathbf{X}_0 - \overline{s}^*)\mathbbm{1}_A = (Y_0 - \boldsymbol{\beta}^{*\top} \mathbf{X}_0 - \overline{s}^*)_+$ almost surely due to the fact that
    \begin{align*}
        (Y_0 - \boldsymbol{\beta}^{*\top} \mathbf{X}_0 - \overline{s}^*)_+ &= (Y_0 - \boldsymbol{\beta}^{*\top} \mathbf{X}_0 - \overline{s}^*)\mathbbm{1}_{\{Y_0 - \boldsymbol{\beta}^{*\top} \mathbf{X}_0 > \overline{s}^*\}}\\
        &\le (Y_0 - \boldsymbol{\beta}^{*\top} \mathbf{X}_0 - \overline{s}^*)\mathbbm{1}_A \\
        &\le (Y_0 - \boldsymbol{\beta}^{*\top} \mathbf{X}_0 - \overline{s}^*)\mathbbm{1}_{\{Y_0 - \boldsymbol{\beta}^{*\top} \mathbf{X}_0 \ge \overline{s}^*\}} \\
        &= (Y_0 - \boldsymbol{\beta}^{*\top} \mathbf{X}_0 - \overline{s}^*)_+.
    \end{align*}
    Similarly, for the complement $A^c$, we have $(Y_0 - \boldsymbol{\beta}^{*\top} \mathbf{X}_0 - \overline{s}^*)\mathbbm{1}_{A^c} = -(Y_0 - \boldsymbol{\beta}^{*\top} \mathbf{X}_0 - \overline{s}^*)_-$. By the construction of $(\mathbf{X}^*,Y^*)$, we have
    \begin{align*}
        Y^* - \boldsymbol{\beta}^{*\top} \mathbf{X}^* - s^* =& Y_0 - \boldsymbol{\beta}^{*\top} \mathbf{X}_0 - s^* + \varepsilon c_{\alpha,p}^{1-q}\left(\alpha^{q-1}\mathbbm{1}_{A}-(1-\alpha)^{q-1}\mathbbm{1}_{A^c}\right)\|\overline{\boldsymbol{\beta}^*}\|_*\\
        = &Y_0 - \boldsymbol{\beta}^{*\top} \mathbf{X}_0 - \overline{s}^* - \frac{\varepsilon}{q} \left(\alpha^{q}-(1-\alpha)^{q}\right) c_{\alpha,p}^{1-q}\|\overline{\boldsymbol{\beta}^*}\|_*\\
        &\qquad + \varepsilon c_{\alpha,p}^{1-q}\left(\alpha^{q-1}\mathbbm{1}_{A}-(1-\alpha)^{q-1}\mathbbm{1}_{A^c}\right)\|\overline{\boldsymbol{\beta}^*}\|_*\\
        =& Y_0 - \boldsymbol{\beta}^{*\top} \mathbf{X}_0 - \overline{s}^* + \varepsilon c_{\alpha,p}^{1-q}\|\overline{\boldsymbol{\beta}^*}\|_*\\
        &\qquad \times \left(\alpha^{q-1}\mathbbm{1}_{A}-(1-\alpha)^{q-1}\mathbbm{1}_{A^c} - \frac{1}{q} \left(\alpha^{q}-(1-\alpha)^{q}\right)\right).
    \end{align*}

    We can see that $Y^* - \boldsymbol{\beta}^{*\top} \mathbf{X}^* - s^*$ is positive on the event $A$ and negative on $A^c$, since $Y_0 - \boldsymbol{\beta}^{*\top} \mathbf{X}_0 - \overline{s}^*$ is positive on $A$ and negative on $A^c$ from the above discussions, and the last term has a non-negative coefficient multiplying a term that is positive on $A$ and negative on $A^c$. Thus, we have
    \begin{align*}
        &\E^{F^*} \left[ \ell_{\alpha}(Y - \boldsymbol{\beta}^{*\top} \mathbf{X}-s^*) \right] \\
        = &\E \left[ \alpha \left(Y^* - \boldsymbol{\beta}^{*\top} \mathbf{X}^* - s^*\right)_+ + (1-\alpha)\left(Y^* - \boldsymbol{\beta}^{*\top} \mathbf{X}^* - s^*\right)_- \right]\\
        = &\E \left[\alpha \left(Y^* - \boldsymbol{\beta}^{*\top} \mathbf{X}^* - s^*\right)\mathbbm{1}_{A} + (\alpha-1)\left(Y^* - \boldsymbol{\beta}^{*\top} \mathbf{X}^* - s^*\right)\mathbbm{1}_{A^c} \right]\\
        = &\E \Bigg[\alpha \left(Y_0 - \boldsymbol{\beta}^{*\top} \mathbf{X}_0 - \overline{s}^* \right)\mathbbm{1}_{A} + (\alpha-1)\left(Y_0 - \boldsymbol{\beta}^{*\top} \mathbf{X}_0 - \overline{s}^* \right)\mathbbm{1}_{A^c} \\
        &\quad + \varepsilon c_{\alpha,p}^{1-q}\|\overline{\boldsymbol{\beta}^*}\|_* \left(\alpha^{q}\mathbbm{1}_{A} + (1-\alpha)^{q}\mathbbm{1}_{A^c} - \frac{1}{q} \left(\alpha^{q}-(1-\alpha)^{q}\right)\left(\alpha\mathbbm{1}_{\{U\ge \alpha\}}+(\alpha-1)\mathbbm{1}_{A^c}\right) \right) \Bigg]\\
        = &\E \left[\alpha \left(Y_0 - \boldsymbol{\beta}^{*\top} \mathbf{X}_0 - \overline{s}^* \right)_+ + (1-\alpha)\left(Y_0 - \boldsymbol{\beta}^{*\top} \mathbf{X}_0 - \overline{s}^* \right)_- \right] + \varepsilon c^{1-q}_{\alpha,p}\|\overline{\boldsymbol{\beta}^*}\|_* \left(\alpha^q(1-\alpha)+(1-\alpha)^q\alpha\right)\\
        = &\E^{F_0} \left[ \ell_{\alpha}(Y - \boldsymbol{\beta}^{*\top} \mathbf{X}-\overline{s}^*) \right] + c_{\alpha,p}\varepsilon\|\overline{\boldsymbol{\beta}^*}\|_*.
    \end{align*}
    The third equality uses the above discussions. The fourth equality holds due to the fact that $\E \left[\mathbbm{1}_{A}\right] = 1-\alpha$ and $\E \left[\mathbbm{1}_{A^c}\right] = \alpha$. The last equality follows from the definition of $c_{\alpha,p}$ in \eqref{eq:C_alphap}. This completes the proof.
\end{proof}

\subsection{Proofs of Section  \ref{sec:3}}
\subsubsection{Proof of Theorem \ref{prop:finite-sample}.}
 The proof of Theorem \ref{prop:finite-sample} relies on the following result coming from Theorem 2 of \citet{wu2022generalization}; see also \cite{ORVW22}.
\begin{lemma} \label{lem:2}
        Given $p\in[1,\infty]$, $\eta\in(0,1)$ and $\mathcal{D}\subseteq \R^{d}$, suppose that $\Gamma:=\E^{F^*}[\|\boldsymbol{\zeta}\|^s]<\infty$ for some $s>2$ and the risk measure $\rho$ satisfies
        \begin{align}
            \label{eq:loss-condition}
            \sup _{|V| \leq \lambda \varepsilon}\rho(Z+V) \geq \sup _{\mathbb{E}[|V|] \leq \varepsilon} \rho(Z+V), \quad \forall Z \in L^1, \varepsilon \geq 0 ,
        \end{align}
        for some $\lambda\geq 1$. Then, we have
        \begin{align*}
            \P\left(\rho^{F^*}\left(\boldsymbol{\beta}^\top \boldsymbol{\zeta}\right) \leq \sup_{F\in \mathbb B_p(\widehat{F}_N,\varepsilon_{N}^0(\eta))}\rho^F\left(\boldsymbol{\beta}^\top \boldsymbol{\zeta}\right),\forall  \boldsymbol{\beta} \in \mathcal{D}\right) \geq 1-\eta,
        \end{align*}
        where $\varepsilon_N^0(\eta)= c_{1/2} \lambda \log(2N+1)^{{1}/{s}}/{\sqrt{N}}$ and  $ c_{1/2} $ is defined by \eqref{eq:c-alpha}.

    \end{lemma}
\begin{proof}[Proof of Theorem  \ref{prop:finite-sample}.]
By Lemma \ref{lem:2}, it suffices to show that the expected quantile loss function $\rho(Z):=\E[\ell_\alpha(Z)]$ satisfies condition \eqref{eq:loss-condition} with $\lambda = (\alpha\vee(1-\alpha))/(\alpha\wedge(1-\alpha))$. While the result has been established in Example 1 of \citet{wu2022generalization}, we provide a complete proof here for the sake of self-containedness. Let $Z \in L^1, \varepsilon \geq 0$, we have
    \begin{align*}
        \sup _{|V| \leq \varepsilon}\E\left[\ell_\alpha(Z+V)\right] &\geq \E\left[\ell_\alpha(Z+\varepsilon\ \sgn(Z))\right] \\
        &= \E\left[\alpha(Z+\varepsilon\ \sgn(Z))_+ + (1-\alpha)(Z+\varepsilon\ \sgn(Z))_-\right] \\
        &= \E\left[\alpha Z_+ + (1-\alpha)Z_- + \varepsilon\left(\alpha\mathbbm{1}_{\{Z<0\}} + (1-\alpha)\mathbbm{1}_{\{Z>0\}}\right)\right] \\
        &\geq \E\left[\ell_\alpha(Z)\right] + \varepsilon(\alpha\wedge(1-\alpha)).
    \end{align*}
    On the other hand, by the Lipschitz continuity of $\ell_\alpha$, we have
    \begin{align*}
        \sup _{\mathbb{E}[|V|] \leq \varepsilon} \E\left[\ell_\alpha(Z+V)\right] &\leq \sup _{\mathbb{E}[|V|] \leq \varepsilon} \E\left[\ell_\alpha(Z)+|V|(\alpha\vee(1-\alpha))\right] = \E\left[\ell_\alpha(Z)\right] + \varepsilon(\alpha\vee(1-\alpha)).
    \end{align*}
    Combining the two inequalities above gives the desired result.
\end{proof}
\subsection{Proofs of Section \ref{sec:4}}
\subsubsection{Proof of Theorem~\ref{prop:finite-sample-fixed}.}
We need the following lemmas to prove Theorem~\ref{prop:finite-sample-fixed}. It is worth noting that the following three lemmas are based on the given  historical observations  $\{(\boldsymbol{x}_i,y_i)\}_{i\in [N]}$ which means that   $\widehat{\boldsymbol{\beta}}_{\rm ols}$
is known.
\begin{lemma} \label{lem:empirical distributions' gap}
    For $p\ge 1$, $\varepsilon>0$ and $\boldsymbol{\beta}\in\R^d$, let $\widehat{F}_N(\boldsymbol{\beta}):=\frac{1}{N}\sum_{i=1}^N \delta_{e_i}$
    with $e_i=y_i-\boldsymbol{\beta}^\top\boldsymbol{x}_i$, $i\in [N]$ and $\widehat{\boldsymbol{\beta}}_N$ be the minimizer of problem \eqref{eq:fixed-design}. We have
\begin{align}\label{eq:empirical distributions' gap}
    \sup_{F \in \mathbb B_p\left(\widehat{F}_N(\boldsymbol{\beta}),\,\varepsilon\right)} \E^{F} \left[\ell_\alpha({e}-\widehat{s}_N)\right] \leq \sup_{F \in \mathbb B_p\left(\widehat{F}_N^{\rm ols},\,\varepsilon\right)} \E^{F} \left[\ell_\alpha(\overline{e}-\widehat{s}_N)\right] +\frac{\Lip(\ell_\alpha)}{N}\sum_{i=1}^{N} \left|(\boldsymbol{\beta}-\betaols)^\top\boldsymbol{x}_i\right|.
\end{align}
\end{lemma}

\begin{proof}
    Since $\ell_\alpha$ is a pointwise supremum of two concave (linear) functions, that is, $\ell_\alpha(u) = \max\{\alpha u, -(1-\alpha)u\}$, and the reference distribution $\widehat{F}_N(\boldsymbol{\beta})$ is discrete, the left-hand side of \eqref{eq:empirical distributions' gap} admits the following finite reduction in the sense that they have the same optimal value (see e.g. \citet{zhen2025unified}):
    \begin{equation}\label{left finite reduction}
        \begin{aligned}
                    \sup\:&\sum_{i=1}^{N} \sum_{k=1}^2 p_{ik}\ell_\alpha(v_{ik}) & \\
        \text{s.t. } &p_{i1}+p_{i2}=\frac{1}{N} & \forall i \in [N],\\
        &\sum_{i=1}^{N} \sum_{k=1}^2 p_{ik} |v_{ik}-e_i+\widehat{s}_N|^p \leq \varepsilon^p & \forall i \in [N],\\
        &p_{ik} \geq 0, v_{ik} \in \R, & \forall i \in [N], k = 1,2.
        \end{aligned}
    \end{equation}
    Note that $e_i=z_i-(\boldsymbol{\beta}-\betaols)^\top\boldsymbol{x}_i$. Substituting this and $v_{ik}$ by $v_{ik}-(\boldsymbol{\beta}-\betaols)^\top\boldsymbol{x}_i$, we have problem  \eqref{left finite reduction} is equivalent to
    \begin{equation}\label{transformed finite reduction}
        \begin{aligned}
            \sup\:&\sum_{i=1}^{N} \sum_{k=1}^2 p_{ik}\ell_\alpha\left(v_{ik}-(\boldsymbol{\beta}-\betaols)^\top\boldsymbol{x}_i\right) & \\
            \text{s.t. } &p_{i1}+p_{i2}=\frac{1}{N} & \forall i \in [N],\\
            &\sum_{i=1}^{N} \sum_{k=1}^2 p_{ik} |v_{ik}-z_i+\widehat{s}_N|^p \leq \varepsilon^p & \forall i \in [N],\\
            &p_{ik} \geq 0, v_{ik} \in \R, & \forall i \in [N], k = 1,2.
        \end{aligned}
    \end{equation}
    By the Lipschitz continuity of $\ell_\alpha$, problem \eqref{transformed finite reduction} is upper bounded by
        \begin{equation} \label{right finite reduction}
        \begin{aligned}
            \sup\:&\sum_{i=1}^{N} \sum_{k=1}^2 p_{ik}\ell_\alpha(v_{ik}) +\Lip(\ell_\alpha)\frac{1}{N}\sum_{i=1}^{N} \left|(\boldsymbol{\beta}-\betaols)^\top\boldsymbol{x}_i\right| & \\
            \text{s.t. } &p_{i1}+p_{i2}=\frac{1}{N} & \forall i \in [N],\\
            &\sum_{i=1}^{N} \sum_{k=1}^2 p_{ik} |v_{ik}-z_i+\widehat{s}_N|^p \leq \varepsilon^p & \forall i \in [N],\\
            &p_{ik} \geq 0, v_{ik} \in \R, & \forall i \in [N], k = 1,2.
        \end{aligned}
    \end{equation}
    By the same finite reduction property, problem  \eqref{right finite reduction} has the same optimal value as the right hand side of \eqref{eq:empirical distributions' gap}. This completes the proof.
\end{proof}

\begin{lemma}\label{lem:true distributions' gap}
For $p\ge 1$ and $\boldsymbol{c}\in\R^d$,    let $F^*_e$  denote the data generating distribution of $e$, and  $F^*_{\overline{e}}$ be the distribution of $\overline{e}:=e+(\boldsymbol{\beta}-\betaols)^\top\boldsymbol{c} $. Then, we have
    \begin{align*}
        \E^{F^*_{\overline{e}}}[\ell_\alpha(\overline{e}-\widehat{s}_N)] \leq \E^{F^*_e}[\ell_\alpha(e-\widehat{s}_N)]+\Lip(\ell_\alpha)\left|(\boldsymbol{\beta}-\betaols)^\top\boldsymbol{c}\right|.
    \end{align*}
\end{lemma}

\begin{proof}
    By the definition of $\overline{e}$, we have
    \begin{align*}
        \E^{F^*_{\overline{e}}}[\ell_\alpha(\overline{e}-\widehat{s}_N)] &= \E^{F^*_{e}}[\ell_\alpha(e+(\boldsymbol{\beta}-\betaols)^\top\boldsymbol{c}-\widehat{s}_N)]\\
        &\leq \E^{F^*_e}[\ell_\alpha(e-\widehat{s}_N)]+\Lip(\ell_\alpha)\left|(\boldsymbol{\beta}-\betaols)^\top\boldsymbol{c}\right|.
    \end{align*}
    where the last inequality follows from the Lipschitz continuity of $\ell_\alpha$. This completes the proof.
\end{proof}

\begin{lemma}\label{lem:constant_radius}
    For   $F_0\in\mathcal M(\R)$ and a constant $A\geq 0$, let $\varepsilon_A = {A}/({\alpha\wedge(1-\alpha)})$. We have
    \begin{align*}
    \E^{F_0}[\ell_\alpha({e}-\widehat{s}_N)] + A \leq \sup_{F \in \mathbb B_p\left(F_0,\varepsilon_A\right)} \E^{F} \left[\ell_\alpha({e}-\widehat{s}_N)\right].
    \end{align*}
\end{lemma}

\begin{proof}
    Given $F_0\in\mathcal M(\R)$, let $e\sim F_0$ and define
    \begin{align*}
        e^\prime=e+\alpha^{-1}A\mathbbm{1}_{\{e\geq \widehat{s}_N\}}-(1-\alpha)^{-1}A\mathbbm{1}_{\{e< \widehat{s}_N\}}.
    \end{align*}
    Denote by $F^\prime$ the distribution of $e^\prime$. Then
    \begin{align*}
        \E^{F^\prime}[\ell_\alpha({e^\prime}-\widehat{s}_N)] &= \E^{F^\prime}[\alpha({e^\prime}-\widehat{s}_N)_+ + (1-\alpha)({e^\prime}-\widehat{s}_N)_-] \\
        &= \E^{F_0}[\alpha({e}-\widehat{s}_N)_+ + (1-\alpha)({e}-\widehat{s}_N)_-] +A\left(\P(e\geq \widehat{s}_N)+\P(e< \widehat{s}_N)\right) \\
        &= \E^{F_0}[\ell_\alpha({e}-\widehat{s}_N)] + A.
    \end{align*}
  Moreover, note that the Wasserstein distance between $F_0$ and $F^\prime$ is bounded by
    \begin{align*}
        W_p(F_0,F^\prime) \leq \left(\E\left[|e-e^\prime|^p\right]\right)^{1/p}\leq \frac{A}{\alpha\wedge(1-\alpha)}=\varepsilon_A.
    \end{align*}
    This completes the proof.
\end{proof}

\begin{lemma}[\cite{cardinali2013observation}] \label{lem:projection_matrix_properties}
Let $h_{i i}$ denote the $i$th diagonal element of the projection matrix $\boldsymbol{H}:=\boldsymbol{X}(\boldsymbol{X}^\top\boldsymbol{X})^{-1}\boldsymbol{X}^\top$. Then, the following results hold:
\begin{enumerate}
    \item [(i)]The diagonal elements satisfy: $0 \leq h_{i i} \leq 1$, $i\in [N]$.
    \item [(ii)]The trace of $\boldsymbol{H}$ equals to the rank of the design matrix, i.e., $\operatorname{trace}(\boldsymbol{H})=\sum_{i=1}^N h_{i i}=d$.
\end{enumerate}
\end{lemma}

Now we are ready to prove Theorem~\ref{prop:finite-sample-fixed}.
\begin{proof}[Proof of Theorem~\ref{prop:finite-sample-fixed}]
Note that $F^*_e$ and $\widehat{F}_N(\boldsymbol{\beta})$ denote the data generating distribution and the empirical distribution of $e$, respectively, and that $e$ satisfies the moment condition of Lemma~\ref{lem:2}. It follows that for any $\eta^{\prime}\in(0,1)$
\begin{align*}
    \P\left(\E^{F^*_{e}}\left[\ell_\alpha({e}-s)\right] \leq \sup_{F \in \mathbb B_p\left(\widehat{F}_N(\boldsymbol{\beta}),\,\varepsilon_{1,N}\right)} \E^{F} \left[\ell_\alpha({e}-s)\right],~\forall~s\in\R\right)\geq 1-\eta^{\prime},
\end{align*}
and thus,
\begin{align} \label{eq:proof-thm3-1}
    \P\left(\E^{F^*_{e}}\left[\ell_\alpha({e}-\widehat{s}_N)\right] \leq \sup_{F \in \mathbb B_p\left(\widehat{F}_N(\boldsymbol{\beta}),\,\varepsilon_{1,N}\right)} \E^{F} \left[\ell_\alpha({e}-\widehat{s}_N)\right]\right) \geq 1-\eta^{\prime},
\end{align}
where $\varepsilon_{1,N}:= \varepsilon_N(\eta^\prime) =  c_\alpha  \log(2N+1)^{{1}/{s}}/{\sqrt{N}} $ and $c_\alpha $ is defined by \eqref{eq:c-alpha} with $\Gamma$ replaced by $\Gamma_0$. Applying Lemma \ref{lem:empirical distributions' gap} and Lemma \ref{lem:true distributions' gap}  to the right hand side and left-hand side of the inequality of the parentheses of \eqref{eq:proof-thm3-1}, respectively, yields that

\begin{align}\label{eq:proof-thm3-2}
    \P\left(\E^{F^*_{\overline{e}}}[\ell_\alpha(\overline{e}-\widehat{s}_N)] \leq \sup_{F \in \mathbb B_p\left(\widehat{F}_N^{\rm ols},\,\varepsilon_{1,N}\right)} \E^{F} \left[\ell_\alpha(\overline{e}-\widehat{s}_N)\right]+A\right) \geq 1-\eta^{\prime},
\end{align}
with $A:=\Lip(\ell_\alpha)\left[\frac{1}{N}\sum_{i=1}^{N} \left|(\boldsymbol{\beta}-\betaols)^\top\boldsymbol{x}_i\right|+\left|(\boldsymbol{\beta}-\betaols)^\top\boldsymbol{c}\right|\right]$. Then by Lemma~\ref{lem:constant_radius}, we have that for any $F \in \mathbb B_p (\widehat{F}_N^{\rm ols},\,\varepsilon_{1,N} )$,
it holds that
\begin{align*}
 \E^{F} \left[\ell_\alpha(\overline{e}-\widehat{s}_N)\right]+A  & \le \sup_{G\in \mathbb B_p\left(F,\,\varepsilon_A\right)}\E^{F} \left[\ell_\alpha(\overline{e}-\widehat{s}_N)\right]\\
 &\le\sup_{G\in \mathbb B_p\left(\widehat{F}_N^{\rm ols},\,\varepsilon_{1,N}+\varepsilon_A\right)}\E^{G} \left[\ell_\alpha(\overline{e}-\widehat{s}_N)\right],
\end{align*}
where
\begin{align*}
    \varepsilon_A :=\frac{A}{\alpha\wedge(1-\alpha)}= \frac{\alpha\vee(1-\alpha)}{\alpha\wedge(1-\alpha)}\left(\frac{1}{N}\sum_{i=1}^{N} \left|(\boldsymbol{\beta}-\betaols)^\top\boldsymbol{x}_i\right|+\left|(\boldsymbol{\beta}-\betaols)^\top\boldsymbol{c}\right|\right),
\end{align*}
and the second inequality follows from $W_p(\widehat{F}_N^{\rm ols},G)\le \varepsilon_{1,N}+\varepsilon_A$ whenever $F \in \mathbb B_p(\widehat{F}_N^{\rm ols},\,\varepsilon_{1,N} )$ and $G\in \mathbb B_p\left(F,\,\varepsilon_A\right)$. It then follows that  $$\sup_{F \in \mathbb B_p\left(\widehat{F}_N^{\rm ols},\,\varepsilon_{1,N}\right)} \E^{F} \left[\ell_\alpha(\overline{e}-\widehat{s}_N)\right]+A\le \sup_{G\in \mathbb B_p\left(\widehat{F}_N^{\rm ols},\,\varepsilon_{1,N}+\varepsilon_A\right)}\E^{G} \left[\ell_\alpha(\overline{e}-\widehat{s}_N)\right].$$ Substituting this into \eqref{eq:proof-thm3-2} yields that
\begin{align}\label{eq:proof-thm3-6}
    \P\left(\E^{F^*_{\overline{e}}}[\ell_\alpha(\overline{e}-\widehat{s}_N)] \leq \sup_{F \in \mathbb B_p\left(\widehat{F}_N^{\rm ols},\,\varepsilon_{1,N}+\varepsilon_A\right)} \E^{F} \left[\ell_\alpha(\overline{e}-\widehat{s}_N)\right]\right) \geq 1-\eta^{\prime}.
\end{align}
It remains to show that $\varepsilon_A$ is bounded by $\epsilon_{1,N}+\epsilon_{2,N}$ with probability at least $1-2\eta^{\prime}$ and then the result holds by taking $\eta=3\eta'$. This can be proved by the same spirit  as in the proofs of Lemma 6 and Lemma 7 in \citet{qi2021robust}. Here we provide a complete proof here using our notations for the sake of self-containedness.

For the first term, let $V_i=(\boldsymbol{\beta}-\betaols)^\top\boldsymbol{x}_i=\boldsymbol{x}_i^\top(\boldsymbol{X}^\top\boldsymbol{X})^{-1}\boldsymbol{X}^\top{\bf e}$, $i\in [N]$ with the vector ${\bf e}=(e_1,\ldots,e_N)$ and $e_i=y_i-\boldsymbol{\beta}^\top \boldsymbol{x_i}$, $i\in [N]$. Since $\boldsymbol{X}$ is fixed, we have
\begin{align*}
\E[V_i^2]&=\boldsymbol{x}_i^\top(\boldsymbol{X}^\top\boldsymbol{X})^{-1}\boldsymbol{X}^\top\E[{\bf e}{\bf e}^\top]\boldsymbol{X}(\boldsymbol{X}^\top\boldsymbol{X})^{-1}\boldsymbol{x}_i\\ &=\E[e^2]\boldsymbol{x}_i^\top(\boldsymbol{X}^\top\boldsymbol{X})^{-1}\boldsymbol{x}_i\\
    &\leq h_{ii}\Gamma_0^{2/s},
\end{align*}
where the second equality is due to the independence and zero-mean of $e_i$, and the last inequality follows from Jensen's inequality. It follows that for any $\varepsilon>0$,
\begin{align}
\mathbb{P}\left(\frac{1}{N} \sum_{i=1}^N\left|V_i\right|>\varepsilon\right) & \leq \frac{\mathbb{E}\left[\frac{1}{N} \sum_{i=1}^N\left|V_i\right|\right]}{\varepsilon} \notag \\
& \leq \frac{\frac{1}{N} \sum_{i=1}^N \sqrt{\E V_i^2}}{\varepsilon} \notag \\
& \leq \frac{\sqrt{\frac{1}{N} \sum_{i=1}^N \E V_i^2}}{\varepsilon} \notag \\
& \leq \frac{\sqrt{\frac{1}{N} \sum_{i=1}^N h_{i i}\Gamma_0^{2/s}}}{\varepsilon} \notag \\
& \leq \sqrt{\frac{d\Gamma_0^{2/s}}{N}} \frac{1}{\varepsilon}, \label{eq:proof-thm3-3}
\end{align}
where the first inequality follows from Markov's inequality, the second and third inequalities are due to Jensen's inequality, and the last
inequality follows from Lemma~\ref{lem:projection_matrix_properties}.

For the second term in $\varepsilon_A$, similarly we have
\begin{align*}
    \E\left[\left|(\boldsymbol{\beta}-\betaols)^\top\boldsymbol{c}\right|^2\right] &= \E\left[\left(\boldsymbol{c}^\top(\boldsymbol{X}^\top\boldsymbol{X})^{-1}\boldsymbol{X}^\top{\bf e}\right)^2\right] \\
    &= \E[e^2]\boldsymbol{c}^\top(\boldsymbol{X}^\top\boldsymbol{X})^{-1}\boldsymbol{c} \\
    &\leq \Gamma_0^{2/s}\boldsymbol{c}^\top(\boldsymbol{X}^\top\boldsymbol{X})^{-1}\boldsymbol{c}.
\end{align*}
Thus, by Chebyshev's inequality, we have for any $\varepsilon>0$,
\begin{align}
    \mathbb{P}\left(\left|(\boldsymbol{\beta}-\betaols)^\top\boldsymbol{c}\right|>\varepsilon\right) & \leq \frac{\mathbb{E}\left[\left|(\boldsymbol{\beta}-\betaols)^\top\boldsymbol{c}\right|^2\right]}{\varepsilon^2}\leq\frac{\Gamma_0^{2/s}\boldsymbol{c}^\top(\boldsymbol{X}^\top\boldsymbol{X})^{-1}\boldsymbol{c}}{\varepsilon^2}. \label{eq:proof-thm3-4}
\end{align}
Now it is sufficient to choose
\begin{align*}
    \varepsilon_{2,N}=\frac{1}{\eta^\prime}\sqrt{\frac{d\Gamma_0^{2/s}}{N}} \quad\text{and}\quad \varepsilon_{3,N}=\sqrt{\frac{\Gamma_0^{2/s}\boldsymbol{c}^\top(\boldsymbol{X}^\top\boldsymbol{X})^{-1}\boldsymbol{c}}{\eta^\prime}},
\end{align*}
so that
\begin{align}\label{eq:proof-thm3-5}
    \P\left(\varepsilon_A\leq \frac{\alpha\vee(1-\alpha)}{\alpha\wedge(1-\alpha)}(\varepsilon_{2,N}+\varepsilon_{3,N})\right)
    & =1-\P\left(   \frac{1}{N} \sum_{i=1}^N\left|V_i\right| + \left|(\boldsymbol{\beta}-\betaols)^\top\boldsymbol{c}\right|  > \varepsilon_{2,N}+\varepsilon_{3,N} \right) \notag\\
    &\geq 1-\P\left(\frac{1}{N} \sum_{i=1}^N\left|V_i\right|>\varepsilon_{2, N}\right)-\P\left(\left|(\boldsymbol{\beta}-\betaols)^\top\boldsymbol{c}\right|>\varepsilon_{3, N}\right)\notag\\
    &\geq 1-\sqrt{\frac{d\Gamma_0^{2/s}}{N}} \frac{1}{\varepsilon_{2, N}}-\frac{\Gamma_0^{2/s}\boldsymbol{c}^\top(\boldsymbol{X}^\top\boldsymbol{X})^{-1}\boldsymbol{c}}{\varepsilon_{3, N}^2}\notag\\
    &\geq 1-2\eta^\prime,
\end{align}
where the second and third inequalities follow from \eqref{eq:proof-thm3-3} and \eqref{eq:proof-thm3-4}, respectively.
Let $N\to\infty$, it is straightforward to verify that $\varepsilon_{2,N}\to 0$, and $\varepsilon_{3,N} = O(1/\sqrt{N})$ as  $N\to\infty$ because
\begin{align*}
  \boldsymbol{c}^\top\left(\boldsymbol{X}^\top \boldsymbol{X}\right)^{-1} \boldsymbol{c}  =   \frac{ \boldsymbol{c}^\top\left(N^{-1} \boldsymbol{X}^\top \boldsymbol{X}\right)^{-1}  \boldsymbol{c}}{N} \approx \frac{ \boldsymbol{c}^\top\Sigma_0^{-1}  \boldsymbol{c}}{N}  ,
\end{align*}
where $\Sigma_0$ is the limit of  $\boldsymbol{X}^\top \boldsymbol{X}/N$ as $X\to\infty$  and the last equality follows from assumption (ii).

Finally, for any $\eta\in(0,1)$, by choosing $\eta^\prime=\eta/3$, we have
\begin{align*}
    \P\left(J_{\rm oos} \leq \widehat{J}_N\right)
    &=1-\P\left(\E^{F^*_{\overline{e}}}[\ell_\alpha(\overline{e}-\widehat{s}_N)] > \sup_{F \in \mathbb B_p\left(\widehat{F}_N^{\rm ols},\,\varepsilon_N\right)} \E^{F} \left[\ell_\alpha({e}-\widehat{s}_N)\right]\right) \\
      &\ge 1-\P\left(
      \begin{array}{c}
           \E^{F^*_{\overline{e}}}[\ell_\alpha(\overline{e}-\widehat{s}_N)] > \sup_{F \in \mathbb B_p\left(\widehat{F}_N^{\rm ols},\,\varepsilon_{1,N}+\varepsilon_A\right)} \E^{F} \left[\ell_\alpha({e}-\widehat{s}_N)\right] \\[3pt]
           ~{\rm or}~\varepsilon_A>  \frac{\alpha\vee(1-\alpha)}{\alpha\wedge(1-\alpha)}(\varepsilon_{2,N}+\varepsilon_{3,N})
      \end{array}
      \right)\\
    &\geq 1-\P\left(\E^{F^*_{\overline{e}}}[\ell_\alpha(\overline{e}-\widehat{s}_N)] > \sup_{F \in \mathbb B_p\left(\widehat{F}_N^{\rm ols},\,\varepsilon_{1,N}+\varepsilon_A\right)} \E^{F} \left[\ell_\alpha({e}-\widehat{s}_N)\right]\right)\\
    &\qquad - \P\left(\varepsilon_A> \frac{\alpha\vee(1-\alpha)}{\alpha\wedge(1-\alpha)}(\varepsilon_{2,N}+\varepsilon_{3,N})\right)\\
    &\geq 1-3\eta^\prime = 1-\eta,
\end{align*}
where the first inequality follows from that  $\varepsilon_{N}\ge \varepsilon_{1,N}+\varepsilon_A$
when $\varepsilon_A \le \frac{\alpha\vee(1-\alpha)}{\alpha\wedge(1-\alpha)}(\varepsilon_{2,N}+\varepsilon_{3,N})$, and thus, $$\sup_{F \in \mathbb B_p\left(\widehat{F}_N^{\rm ols},\,\varepsilon_{N}\right)} \E^{F} \left[\ell_\alpha({e}-\widehat{s}_N)\right]\ge \sup_{F \in \mathbb B_p\left(\widehat{F}_N^{\rm ols},\,\varepsilon_{1,N}+\varepsilon_A\right)} \E^{F} \left[\ell_\alpha({e}-\widehat{s}_N)\right]; $$
and the third inequality follows from \eqref{eq:proof-thm3-6} and \eqref{eq:proof-thm3-5}.
This completes the proof.
\end{proof}
\end{appendices}

\end{document}